\documentclass[12pt,reqno]{article}
\usepackage[utf8]{inputenc}
\usepackage{amsmath}
\usepackage{amsfonts}
\usepackage{amssymb}
\usepackage{amsthm}	
\usepackage{cases}
\usepackage{subeqnarray}
\usepackage{cite}
\usepackage[colorlinks]{hyperref}
\usepackage{graphics}
\usepackage{graphicx}
\usepackage{fancyref}
\usepackage{epstopdf}
\usepackage{float}
\usepackage{enumerate} 
\usepackage{bm}
\usepackage{titlesec}
\usepackage[titletoc,toc,title]{appendix}

\makeatletter
\allowdisplaybreaks \numberwithin{equation}{section}
\newtheorem{theorem}{Theorem}[section]

\newtheorem{lemma}{Lemma}[section]
\newtheorem{proposition}{Proposition}[section]

\newtheorem{definition}{Definition}[section]
\newtheorem{remark}{Remark}[section]

\newtheorem{thmx}{Theorem}

\begin{document}

\title{Principal spectral theory and asymptotic behavior of the spectral bound for partially degenerate nonlocal dispersal systems
\date{\empty}
\date{\empty}
\author{ Lei Zhang \thanks{The corresponding author, E-mail: zhanglei890512@gmail.com} \\
	{\small School of Mathematics and Statistics, Shaanxi Normal University,}\\
	{\small Xi'an, Shaanxi 710119, China.}\\
}
}\maketitle
\begin{abstract}
The purpose of this paper is to investigate the principal spectral theory and asymptotic behavior of the spectral bound for cooperative nonlocal dispersal systems, specifically focusing on the case where partial diffusion coefficients are zero, referred to as the partially degenerate case. We propose two sufficient conditions that ensure the existence of the principal eigenvalue in these partially degenerate systems. Additionally, we study the asymptotic behavior of the spectral bound for nonlocal dispersal operators with small and large diffusion coefficients, considering both non-degenerate and partially degenerate cases. Notably, we find a threshold-type result as the diffusion coefficients tend towards infinity in the partially degenerate case. Finally, we apply these findings to discuss the asymptotic behavior of the basic reproduction ratio in a viral diffusion model.

	
\end{abstract}

\section{Introduction}\label{sec:intro}

In this paper, we investigate the principal eigenvalue theory and the asymptotic behavior of the spectral bound of the nonlocal dispersal systems with the form
\begin{equation}\label{equ:eig:main}
	d_i\left[\int_{\Omega} k_i(x,y) u_i(y) \mathrm{d} y- \int_{\Omega} k_i(y,x)u_i(x) \mathrm{d} y \right]
	+\sum_{j=1}^{l}m_{ij}(x)u_j(x)= \lambda u_i(x), ~x\in \overline{\Omega}, ~1\leq i \leq l.
\end{equation}
Here $d_i \geq 0$ is the diffusion coefficient; $\Omega$ is an open bounded domain in $\mathbb{R}^N$;  $k_i(x,y)$ is a non-negative continuous function of $(x,y)\in \overline{\Omega} \times \overline{\Omega}$; $l$ is a positive integer; $M(x)=(m_{ij}(x))_{l \times l}$ with $m_{ij}(x)$ being a continuous function on $\overline{\Omega}$. So it is easy to see that $M \in C(\overline{\Omega}, \mathbb{R}^{l \times l})$. Let $l_1$ and $l_2$ be two positive integers with $l_1 +l_2=l$. Clearly,
$$
\int_{\Omega} \int_{\Omega} k_i(x,y)v(y) \mathrm{d} y \mathrm{d} x
=\int_{\Omega} \int_{\Omega} k_i(y,x)v(x) \mathrm{d} y \mathrm{d} x,~ \forall v \in C(\overline{\Omega}),~1\leq i \leq l,
$$ 
which implies that the total particles are conserved in the dispersal process. This corresponds to the Neumann boundary condition (see, e.g., \cite{li2017eigenvalue,andreu2010nonlocal}).

Recall that a square matrix is said to be cooperative if its off-diagonal elements are nonnegative, and nonnegative if all elements are nonnegative; a square matrix is said to be irreducible if it is not similar, via a permutation, to a block lower triangular matrix, and reducible if otherwise. We have the following assumptions.
\begin{enumerate}
	\item[(H1)]  For any $x \in \overline{\Omega}$, $M(x)$ is cooperative. 
	\item[(H2)] For any $x \in \overline{\Omega}$, $M(x)$ is irreducible.
	\item[(H3)] For any $x \in \overline{\Omega}$, $1 \leq i \leq l$, $k_i(x,x)>0$.
	\item[(H4)] For each $1 \leq i \leq l$, $d_i>0$.
	\item[(H4$'$)] For each $1 \leq i \leq l_1$, $d_i>0$, and for each $l_1 +1 \leq i \leq l$, $d_i=0$.
\end{enumerate}

We call the system non-degenerate if all diffusion coefficients are positive, and partially degenerate if some diffusion coefficients are zero. Here (H4) and (H4$'$) corresponds the non-degenerate and partially degenerate cases, respectively. Some stream population models such as benthic-drift models and pulsed bio-reactor models with a hydraulic storage zone are described by partially degenerate reaction-diffusion systems (see, e.g.,\cite{pachepsky2005persistence,lutscher2006effects,jin2011seasonal,hsu2011dynamics,mckenzie2012r_0,hsu2015pivotal,huang2016r0,jiang2004saddle}). Recently, partially degenerate nonlocal dispersal systems are beginning to be used in modeling such problems (see, e.g., \cite{wang2018global,bao2020propagation}).

In recent years, significant attention has been given to the research of nonlocal dispersal systems (see, e.g., \cite{andreu2010nonlocal,fife2003some,du2012analysis,kao2010random,hutson2003evolution,coville2008existence,bates2007existence}). It is of great importance to study the eigenvalue problem of the linear evolution systems with nonlocal dispersal. The main challenge to study this problem lies in the lack of compactness for the associated linear operators. Bates and Zhao\cite{bates2007existence}, as well as Coville \cite{coville2010simple} established the principal eigenvalue theory for a scalar nonlocal dispersal equation by a generalized Krein-Rutman theorem, going back to Nussbaum \cite{nussbaum1981eigenvectors} and Edmunds, Potter and Stuart \cite{edmunds1972non}. Meanwhile, Shen and Zhang \cite{shen2010spreading} also developed this eigenvalue problem using a powerful tool developed by B{\"u}rger \cite{burger1988perturbations}, which investigated the eigenvalue problem of perturbations of positive semigroups.
Huston, Shen and Vickers \cite{hutson2008spectral} and Coville \cite{coville2010simple} also proved the continuity of spectral bound and Lipschitz continuity of the principal eigenvalue with respect to parameters, respectively. 
More recently, Li, Coville and Wang \cite{li2017eigenvalue} obtained the estimates for the principal eigenvalue and extended the existence theory. 
Berestycki, Coville and Vo\cite{berestycki2016definition} introduced the generalized principal eigenvalue.

Consider the following principal eigenvalue problem with nonlocal dispersal:
\begin{equation}\label{equ:eig}
	d\left[\int_{\Omega} k(x-y) u(y) \mathrm{d} y
	- \int_{\Omega} k(y-x) u(x) \mathrm{d} y \right]
	+m(x) u(x)= \lambda u(x), ~x\in \overline{\Omega}.
\end{equation}
Here $k(x)$ is a non-negative continuous function of $x\in \overline{\Omega}$ with $\int_{\mathbb{R}} k(x) \mathrm{d} x =1$, $k(x)=k(-x)$ and $k(0)>0$, $\forall x \in \overline{\Omega}$, which is the dispersal kernel; $d>0$ is the diffusion coefficient; $m(x)$ is a continuous function on $\overline{\Omega}$; $u(x)$ is the density function of dispersing particles.
By summarizing the results in \cite{shen2010spreading,coville2010simple,li2017eigenvalue}, the eigenvalue problem \eqref{equ:eig} has the principal eigenvalue $\lambda(d)$ if one of the following statements holds:
\begin{itemize}
	\item There is an open set $\Omega_0 \subset \Omega$ such that $(\max_{x\in\overline{\Omega}} a(x)- a(x))^{-1} \not\in L^1 (\Omega_0)$, where $a(x):=m(x)-\int_{\Omega} k(y-x) \mathrm{d} y$, $\forall x \in \overline{\Omega}$ .
	\item The diffusion coefficient satisfies $d>d^*:= \max_{x\in\overline{\Omega}} m(x) - \min_{x \in \overline{\Omega}} m(x)$.
\end{itemize}

It is natural to ask whether such results can be extended to cooperative systems.
Shen and her collaborators\cite{bao2017criteria,shenxie2016spectral} gave a confirmed answer and extended the principal eigenvalue theory to the time-periodic case. 
Liang, Zhang and Zhao \cite{liang2019principal} developed the principal eigenvalue theory for nonlocal dispersal systems with time-delay in a time-periodic environment. Recently, Su, Wang and Zhang \cite{Su2023JDE} studied the principal spectral theory and variational characterizations for cooperative systems with nonlocal and coupled diffusion without utilizing the generalized Krein-Rutman theorem.  However, it seems that these methods and arguments may not be well adapted to the partially degenerate systems. Wang, Li and Sun \cite{wang2018global} gave a sufficient condition for the existence of the principal eigenvalue problem for the linear evolution systems coupled with a nonlocal dispersal equation and an ordinary differential equation. It is worth pointing out that this eigenvalue problem has not been addressed in general cases. We thus aim to establish the principal eigenvalue theory of the partially degenerate nonlocal dispersal systems.

Since the principal eigenvalue of \eqref{equ:eig:main} cannot be expressed explicitly or even by a variational formula, it is a challenging problem to quantitatively analyze the principal eigenvalue. 
Recently, Yang, Li, and Ruan \cite{yang2019dynamics} discussed the asymptotic behaviors of the principal eigenvalue of \eqref{equ:eig} for small and large diffusion coefficients:
\begin{itemize}
	\item $\lambda(d) \rightarrow \max_{x \in \overline{\Omega}} m(x)$ as $d \rightarrow 0^+$.
	\item $\lambda(d) \rightarrow \frac{1}{\vert \Omega \vert}\int_{\Omega} m(x) \mathrm{d} x$ as $d \rightarrow +\infty$. 
\end{itemize}
One may ask how to characterize the above limiting profiles for cooperative systems.
Such problems have been explored for patch models (see, e.g., \cite{allen2007Asymptotic,gao2019travel,gao2020fast}) and reaction-diffusion systems (see, e.g.,\cite{hale1987varying,dancer2009principal,lam2016asymptotic,allen2008asymptotic}). However, due to the lack of compactness for the associated linear operators, these questions are still open for nonlocal dispersal systems, especially for the partially degenerate case.

It is worth pointing out that the principal eigenvalue may not exist for nonlocal dispersal systems (see, e.g., \cite[Section 6]{shen2010spreading}). As is well known, the principal eigenvalue is a primary tool to obtain profound results in various differential equations, especially in elliptic and parabolic problems. For instance, the principal eigenvalue can be used to analyze the stability of equilibrium in the investigation of spatial dynamics (see, e.g., \cite{ni2011mathematics, cantrell2004spatial, zhao2017dynamical}) and characterize the spreading speed in the research of propagation dynamics (see, e.g., \cite{liang2007asymptotic, berestycki2002front}). When the principal eigenvalue does not exist, the spectral bound can serve as an alternative tool to analyze the stability of equilibrium and characterize the spreading speed for a scalar nonlocal dispersal equation (see, e.g., \cite{shen2012Stationary}). This motivates us to study the limiting profile of the spectral bound of the associated nonlocal dispersal operators with small and large diffusion coefficients for non-degenerate and partially degenerate cases.

The first purpose of this paper is to establish the principal eigenvalue theory of the partially degenerate nonlocal dispersal systems.
For each $1 \leq i \leq l$, write $\chi_i(x):=\int_{\Omega} k_i(y,x) \mathrm{d} y$, $\forall x \in \overline{\Omega}$.
Thanks to the assumption (H3), there exists a positive continuous function $p_i$ on $\overline{\Omega}$ such that $\int_{\Omega} k_i(x,y) p_i(y) \mathrm{d} y = \chi_i (x) p_i(x)$, $\forall x \in \overline{\Omega}$ and $\int_{\Omega} p_i(x) \mathrm{d} x=1$ (see Lemma \ref{lem:p}). 
For any given $\bm{d}:=(d_1,\cdots,d_l)^T$ with $d_i \geq 0$, let $\mathcal{P}(\bm{d})$ be a family of linear operators on $C(\overline{\Omega},\mathbb{R}^l)$ defined by
$$
\mathcal{P}(\bm{d})\bm{u}=
\left(
\mathcal{P}_1(\bm{d})\bm{u},
\cdots,
\mathcal{P}_l(\bm{d})\bm{u}
\right), ~\forall \bm{u}=(u_1,\cdots,u_l)^T \in C(\overline{\Omega},\mathbb{R}^l),
$$
where 
$$
[\mathcal{P}_i(\bm{d})\bm{u}](x):=d_i\left[\int_{\Omega} k_i(x,y) u_i(y) \mathrm{d} y- \chi_i (x)u_i(x) \right]
+\sum_{j=1}^{l}m_{ij}(x)u_j(x), ~1 \leq i \leq l.
$$
\begin{definition}
	The real number $\lambda_*$ is called the principal eigenvalue of $\mathcal{P}(\bm{d})$ if $\lambda_*$ is an isolated eigenvalue of $\mathcal{P}(\bm{d})$ with finite algebraic multiplicity corresponding to a positive eigenvector and $\lambda_* \geq \mathrm{Re}\lambda$ for all $\lambda \in \sigma(\mathcal{P}(\bm{d}))$, where $\sigma(\mathcal{P}(\bm{d}))$ is the spectrum set of $\mathcal{P}(\bm{d})$. 
\end{definition}

For each $x \in \overline{\Omega}$, we split the matrix $M(x)$ into 
\begin{equation}\label{equ:M(x)}
	M(x)=
	\left(
	\begin{matrix}
		M_{11}(x)& M_{12}(x)\\
		M_{21}(x)& M_{22}(x)
	\end{matrix}
	\right),
\end{equation}
where $M_{11}(x)$, $M_{12}(x)$, $M_{21}(x)$, $M_{22}(x)$ are $l_1 \times l_1$, $l_1 \times l_2$, $l_2 \times l_1$, $l_2 \times l_2$ matrices, respectively. For any $\gamma > \eta_{22}:= \max_{x\in \overline{\Omega}} s(M_{22}(x))$, let $B_{\gamma}(x)=(b_{ij,\gamma}(x))_{l_1 \times l_1}$ be a continuous $l_1 \times l_1$ matrix-valued function of $x \in \overline{\Omega}$ defined by
$$
B_{\gamma}(x):= M_{11} (x) + M_{12}(x)(\gamma I_2 - M_{22}(x))^{-1} M_{21}(x).
$$
Here $I_2$ is the identity $l_2 \times l_2$ matrix. Define $\tilde{B}_{\gamma}:= (\tilde{b}_{ij,\gamma})_{l_1 \times l_1}$ by $\tilde{b}_{ij,\gamma} = \int_{\Omega} b_{ij,\gamma} (x) p_j(x) \mathrm{d} x$, $\forall 1 \leq i,j \leq l_1$. Let $A(x)$ be a matrix-valued function of $x \in \overline{\Omega}$ defined by
\begin{equation}\label{equ:A(x)}
	A(x):=
	\left(
	\begin{matrix}
		M_{11}(x)-D_1(x)& M_{12}(x)\\
		M_{21}(x)& M_{22}(x)
	\end{matrix}
	\right).
\end{equation}
Here $D_1(x)= {\rm diag}(d_1 \chi_1(x),\cdots,d_{l_1} \chi_{l_1}(x))$ with $d_i>0$, $i=1,\cdots,l_1$. 
Write
$H(x) := s(A(x)),~ \forall x \in \overline{\Omega},$
and 
$\eta:= \max_{x \in \overline{\Omega}} H(x)$. Here, $s(\cdot)$ is the spectral bound of the associated operator, which is defined in Section \ref{sec:pre}. The first main result of this paper is
\begin{thmx}\label{thm:A}{\rm (see Theorems \ref{thm:dene} and \ref{thm:D:large:exist})}
	Assume that {\rm (H1)--(H3)} and {\rm (H4$'$)} hold. Then the following statements are valid.
	\begin{itemize}
		\item[\rm (i)] If there is an open set $\Omega_0 \subset \Omega$ such that $(\eta- H)^{-1} \not\in L^1 (\Omega_0)$, then the principal eigenvalue of $\mathcal{P}(\bm{d})$ exists.
		\item[\rm (ii)] Assume that $\lim\limits_{\gamma \rightarrow \eta_{22}^+} s(\tilde{B}_{\gamma}) > \eta_{22}$. Then there exists $\hat{d}>0$ large enough such that $\mathcal{P}(\bm{d})$ has the principal eigenvalue if $\min_{1\leq i \leq l_1} d_i>\hat{d}$.
	\end{itemize}
\end{thmx}
It is worth pointing out that the sufficient condition given in Theorem \ref{thm:A}(i) is consistent with that for the non-degenerate case in \cite{bao2017criteria}. Moreover, this result generalize the sufficient condition given in \cite{wang2018global}. In this paper, we employ a generalized Krein-Rutman theorem to obtain the above theorem. The main difficulty in carrying out this construction is to find a sufficient condition such that there is a ``gap" between the spectral radius and the essential spectral radius of $\mathcal{P}(\bm{d})$. Although the structure of essential spectral points is unified whether the system is degenerate or not (see, e.g., \cite{coville2010simple} and \cite{liang2017principal}), techniques for showing that the gap exists in the non-degenerate case are not well adapted to the partially degenerate case. To overcome this technical difficulty, inspired by \cite{wang2012basic,hsu2015pivotal,liang2017principal}, we introduce a family of non-degenerate nonlocal dispersal eigenvalue problems and establish a relation between our problem and these problems.

The second purpose of this paper is to study the asymptotic behavior of the spectral bound of the associated nonlocal dispersal operators with small and large diffusion coefficients, both for the non-degenerate case and the partially degenerate case.
Define a cooperative matrix $\tilde{M}:=(\tilde{m}_{ij})_{l\times l}$ by $\tilde{m}_{ij}=\int_{\Omega} m_{ij} (x) p_{j}(x) \mathrm{d} x$, $\forall 1 \leq i,j\leq l$.
Write $\kappa:=\max_{x \in \overline{\Omega}}s(M(x))$ and $\tilde{\kappa}:=s(\tilde{M})$.
Our second main result is
\begin{thmx}\label{thm:B}{\rm (see Theorems \ref{thm:0}, \ref{thm:D:inf:non-degen} and \ref{thm:inf:degen})}
	Assume that {\rm (H1)--(H3)} hold. Then the following statements are valid:
	\begin{itemize}
		\item[\rm (i)]$s(\mathcal{P}(\bm{d})) \rightarrow \kappa$ as $\max_{1\leq i \leq l} d_i \rightarrow 0$.
		\item[\rm (ii)] If, in addition, {\rm (H4)} holds, then $s(\mathcal{P}(\bm{d})) \rightarrow \tilde{\kappa}$ as $\min_{1\leq i \leq l} d_i \rightarrow +\infty$.
		\item[\rm(iii)] If, in addition, {\rm (H4$'$)} holds, then the following threshold results are valid:
		\begin{itemize}
			\item[\rm (a)] If $\lim\limits_{\gamma \rightarrow \eta_{22}^+ } s(\tilde{B}_{\gamma}) >\eta_{22}$, then there exists a unique $\gamma^*> \eta_{22}$ such that $s(\tilde{B}_{\gamma^*}) =\gamma^*$ and 
			$s(\mathcal{P}(\bm{d})) \rightarrow \gamma^* \text{ as } \min_{1\leq i \leq l_1} d_i \rightarrow +\infty.$
			\item[\rm (b)] If $\lim\limits_{\gamma \rightarrow \eta_{22}^+ } s(\tilde{B}_{\gamma}) \leq \eta_{22}$, then
			$s(\mathcal{P}(\bm{d})) \rightarrow \eta_{22} \text{ as } \min_{1\leq i \leq l_1} d_i \rightarrow +\infty.$
		\end{itemize}
	\end{itemize}
\end{thmx}

Combining the structure of essential spectral points, perturbation theory, and the method of upper and lower approximation, we show that the spectral bound is continuous with respect to parameters (see Theorem \ref{thm:conti:KA}). As a straightforward consequence, the asymptotic behavior of the spectral bound as the diffusion coefficients approach zero is characterized (i.e., Theorem \ref{thm:B}(i)). To investigate the limiting profiles of the spectral bound with large diffusion coefficients, we distinguish between two cases. For the non-degenerate case (i.e., Theorem \ref{thm:B}(ii)), we employ ideas from \cite{yang2019dynamics} and \cite{gao2020fast}. To investigate the partially degenerate case (i.e., Theorem \ref{thm:B}(iii)), we combine perturbation techniques, comparison arguments, and methods used in the proof of the non-degenerate case and Theorem \ref{thm:A}.

The remaining part of the paper is organized as follows. In the next section, we provide the definition of the principal eigenvalue and present some preliminary results on positive operators and cooperative matrices. In Section \ref{sec:existence}, we establish the principal eigenvalue theory for partially degenerate linear evolution systems with nonlocal dispersal. In Section \ref{sec:asy}, we investigate the asymptotic behavior of the spectral bound of the associated nonlocal dispersal operators with small and large diffusion coefficients for both the non-degenerate case and the partially degenerate case. In Section \ref{sec:app}, we employ these results to discuss the asymptotic behavior of the basic reproduction ratio for a infection model with cell-to-cell transmission and nonlocal viral dispersal. Additionally, we provide a proof of the continuity of the spectral bound of the associated nonlocal dispersal operators with respect to parameters in the appendix.

\section{Preliminaries}\label{sec:pre}

Let $E$ be a Banach space with a cone $K \subset E$. We denote by $\mathrm{Int} (K)$ the interior of the cone $K$. We also use $\geq$($>$ and $\gg$) to represent the (strict and strong) order relation induced by the cone $K$.
Assume that $A$ is a bounded linear operator on $E$. The operator $A$ is said to be positive if $Ax \geq 0$ for any $ x \geq 0$ and strongly positive if $A x \gg 0$ for any $x >0$. Let $\sigma(A)$ and $\sigma_e(A)$ be the spectrum set and the essential spectrum set of $A$, respectively (for the definition of the essential spectrum, we refer to \cite[Section 7.5]{schechter1971principles}). The spectral bound $s(A)$ and the essential spectral bound $s_e(A)$ are defined by
$$
s(A):= \sup \{\mathrm{Re} \lambda : \lambda \in \sigma (A) \}, ~\text{ and }
s_e(A):= \sup \{\mathrm{Re} \lambda : \lambda \in \sigma_e (A) \}.
$$
We use $r(A)$ and $r_e(A)$ to represent its spectral radius and essential spectral radius of $A$.
\begin{definition}
	Let $E$ be a Banach space with a cone $K \subset E$. Assume that $B$ is a bounded linear operator on $E$ and there exists some $c_0$ such that $B+c_0 I$ is a positive operator, where $I$ is the identity map on $E$. The real number $\lambda_*$ is called the principal eigenvalue of $B$ if $\lambda_*$ is an isolated eigenvalue of $B$ with finite algebraic multiplicity corresponding to a positive eigenvector and $\lambda_* \geq \mathrm{Re}\lambda$ for all $\lambda \in \sigma(B)$. 
\end{definition}

We remark that if $\lambda_*$ is the principal eigenvalue of $B$, then $\lambda_*=s(B)$.
\begin{lemma}\label{lem:spectral:bound_radius}
	Let $E$ be a Banach space with a cone $K \subset E$. If $B$ is a bounded positive linear operator on $E$ and $r(B) \in \sigma(B)$, then $s(B)=r(B)$. 
\end{lemma}

\begin{proof}
	It is obvious that $r(B) \geq s(B)$ by their definitions. Thanks to $r(B) \in \sigma(B)$, we have $r(B) \leq s(B)$, and hence, $s(B)=r(B)$.
\end{proof}
The following results can be derived directly by \cite[Proposition 1]{schaefer1960some} and the above lemma.
\begin{remark}\label{rem:spectral:bound_radius}
	Let $E$ be a Banach space, and $K \subset E$ be a normal cone with nonempty interior(for the definition of the normal cone, we refer to \cite[Section 19]{deimling1985nonlinear}). If $B$ is a bounded positive linear operator on $E$, then $r(B) \in \sigma(B)$, and hence, $s(B)=r(B)$. 
\end{remark}
Next, we review a generalized Krein-Rutman theorem (see, e.g.,\cite[Corollary 2.2]{nussbaum1981eigenvectors}), which provide a primary tool to prove the existence of the principal eigenvalue for a non-compact bounded linear positive operator.
\begin{lemma}\label{lem:gene:K-R}
	Let $E$ be a Banach space with a total cone $K \subset E$ (i.e. $E=\overline{K-K}$). If $B$ is a bounded positive linear operator with $r(B)> r_e(B)$, then $r(B)$ is the principal eigenvalue of $B$. 
\end{lemma}

\begin{lemma}\label{lem:eventually}
	Let $E$ be a Banach space, and $K \subset E$ be a cone with nonempty interior. Assume that $B$ is a bounded positive linear operator on $E$ and $B^{n}$ is strongly positive for some positive integer $n$. If $r(B^{n})$ is the principal eigenvalue of $B^{n}$, then $r(B)$ is the principal eigenvalue of $B$. Moreover, the following statements are valid:
	\begin{itemize}
		\item[\rm(i)] There exists $u \gg 0$ such that $Bu= r(B) u$.
		\item[\rm(ii)] $r(B)$ is an algebraically simple eigenvalue of $B$.
		\item[\rm(iii)] $u$ is the unique non-negative eigenvector of $B$ up to scalar multiplications on $E$.
	\end{itemize}
\end{lemma}
\begin{proof}
	In view of \cite[Appendix A]{liang2017principal} or \cite[Theorem 19.3]{deimling1985nonlinear}, $r(B^n)>0$ holds ture and the statements (i)--(iii) are valid with $B$ replaced by $B^n$. We assume that $n \geq 2$ and let $u \gg 0$ be an eigenvector of $B^n$ corresponding to $r:=r(B^n)$. It then follows from Gelfand's formula(see, e.g., \cite[Theorem VI.6]{reed1980methods}) that $\lambda^n=r$, where $\lambda:=r(B)>0$. 
	
	Motivated by \cite[Lemma 3.1]{liang2007asymptotic}, we first show that $u$ is an eigenvector of $B$ corresponding to $\lambda$. In view of $B^n [Bu]=B[B^nu]=rBu$, it is easy to see that $Bu$ is an eigenvector of $B^n$. Since $r$ is an algebraically simple eigenvalue of $B^n$, we have $Bu= s_0 u \geq 0$ for some $s_0 \neq 0$. Clearly, $Bu>0$ and $s_0>0$, otherwise $Bu = 0$ yields that $B^n u=0$, which contradicts with $r(B^n)>0$. Thus, $B^n u=s_0^n u$ leads to $s_0^n =r$, and hence, $s_0 = \lambda$ and $B u=\lambda u$. Part (i) is obtained. We notice that $r$ is an isolated eigenvalue of $B^n$, then $\lambda$ is that of $B$. Lemma \ref{lem:spectral:bound_radius} yields that $\lambda=s(B)$ is the principal eigenvalue of $B$.
	
	(ii) Noting that $r=\lambda^n$ is an algebraically simple eigenvalue of $B^n$, we first prove that $\lambda$ is a geometrically simple eigenvalue of $B$.
	Indeed, if $B u'=\lambda u'$ for some $u'\neq 0$, then $B^n u' = \lambda^n u'$ implies that $u'$ is a scalar multiplication of $u$. 
	
	To show that $\lambda$ is algebraically simple, we suppose that there exists a pair of $v$ and $t_0$ such that $(\lambda I - B) v = t_0 u$. By the fact that $r=\lambda^n$ is an algebraically simple eigenvalue of $B^n$ again, it then follows that $t_0=0$ and $v$ is a scalar multiplication of $u$ from
	$$
	\begin{aligned}
		(\lambda^n I -B^n) v 
		&=(\lambda^{n-1} I+ \lambda^{n-2}B + \cdots + \lambda B^{n-2} +B^{n-1} )(\lambda I -B)v\\
		&=(\lambda^{n-1} I+ \lambda^{n-2}B + \cdots + \lambda B^{n-2} +B^{n-1} )t_0 u\\
		&=n \lambda^{n-1}t_0 u.
	\end{aligned}
	$$
	
	(iii)
	Suppose that there exists a pair of $w >0 $ and $\mu$ such that $B w = \mu w$. It is easily seen that $B^n w = \mu^n w$. Since $u$ is the unique non-negative eigenvector of $B^n$ up to scalar multiplications on $E$, we then have $\mu^n = r$ and $w$ is a scalar multiplication of $u$.
\end{proof}

In order to facilitate the use of a generalized Krein-Rutman theorem, we provide the following lemma.
\begin{lemma}\label{lem:gene:K-R:s}
	Let $E$ be a Banach space, and $K \subset E$ be a cone with nonempty interior. Assume that $B$ is a bounded linear operator with $s(B)> s_e(B)$ and $B+c_0I$ is positive for some $c_0>0$. Then $s(B)$ is the principal eigenvalue of $B$. If, in addition, $(B+c_0I)^n$ is strongly positive for some integer $n \geq 1$, then the following statements are vaild:
	\begin{itemize}
		\item[\rm(i)] There exists $u \gg 0$ such that $Bu= s(B) u$.
		\item[\rm(ii)] $s(B)$ is an algebraically simple eigenvalue.
		\item[\rm(iii)] $u$ is the unique non-negative eigenvector of $B$ up to scalar multiplications on $E$.
		\item[\rm(iv)] There exists $\delta > 0$ such that 
		$\mathrm{Re} \lambda < s(B) -\delta$ for any $\lambda \in \sigma(B) \setminus \{ s(B) \}$.
	\end{itemize}
\end{lemma}

\begin{proof}
	It is easy to see that 
	\begin{equation}\label{equ:sigma:c}
		\sigma(B+cI) = \{\lambda +c : \lambda \in \sigma(B) \} \text{ and } \sigma_e(B+cI) = \{\lambda +c : \lambda \in \sigma_e(B) \}.
	\end{equation}
	Since $B$ is a bounded linear operator, it then follows that all spectral points of $B$ are bounded. Thus, $s(B)> s_e(B)$ implies that $r(B+c_1 I) > r_e(B+c_1 I)$ for some large constant $c_1$ with $c_1>c_0$ due to \eqref{equ:sigma:c}. Lemma \ref{lem:gene:K-R} yields that $r(B+c_1 I)$ is the principal eigenvalue of $B+ c_1 I$. By Lemma \ref{lem:spectral:bound_radius}, we have $r(B+c_1 I)= s(B+c_1 I)=s(B)+c_1$, and hence, $s(B)$ is the principal eigenvalue of $B$. Since $(B+c_0I)^n$ is strongly positive, so is $(B+c_1I)^n$. Thanks to Lemma \ref{lem:gene:K-R} again, $r((B+c_1 I)^n)$ is the principal eigenvalue of $(B+c_1I)^n$. Thus, (i)--(iii) can be derived directly by Lemma \ref{lem:eventually}. 
	
	In view of $s(B)> s_e(B)$, to prove (iv), it is sufficient to verify that there is no spectral point $\mu \neq s(B)$ such that $\mathrm{Re} \mu\geq s(B)$. Otherwise, $(\mu+c_1)^n$ is still a spectral point of $(B+c_1 I)^n$, whose modulus is greater than $(s(B)+c_1)^n=r((B+c_1 I)^n)$, which is impossible.
\end{proof}

\begin{lemma}\label{lem:positive:zero}
	Let $E$ be a Banach space, $K \subset E$ be a cone with nonempty interior, and $B$ be a bounded positive linear operator on $E$. If $Bu=0$ for some $u \in \mathrm{Int}(K)$, then $B$ is zero operator, that is, $B v=0$ for all $v \in E$.
\end{lemma}
\begin{proof}
	In view of $u \in \mathrm{Int}(K)$, there exists $r>0$ such that $u \geq r v$ for all $v \in E$ with $\Vert v \Vert=1 $. Since $B$ is positive, we have $0=Bu \geq r B v, ~ \forall v \in E$ with $\Vert v \Vert=1 $. Thus, $B v \leq 0$ and $-Bv=B (-v) \leq 0$ imply that $B v=0$ for all $v \in E$.
\end{proof}

\begin{lemma}\label{lem:strongly}
	Let $E$ be a Banach space, and $K \subset E$ be a cone with nonempty interior. Assume that $A$ is a bounded strongly positive linear operator and $B$ is nonzero positive operator on $E$. We further assume that $r(A)$ and $r(A+B)$ are the principal eigenvalue of $A$ and $A+B$, respectively. Then $s(A+B)=r(A+B)>s(A)=r(A)>0.$
\end{lemma}

\begin{proof}
	Lemma \ref{lem:spectral:bound_radius} yields that $s(A+B)=r(A+B)$ and $s(A)=r(A)$.
	According to \cite[Theorem 1.1]{burlando1991monotonicity} and \cite[Theorem 19.3]{deimling1985nonlinear}, it is easy to see that $r(A+B) \geq r(A)>0$. 
	Suppose, by contradiction, that $\lambda:=r(A+B)=r(A)>0$. By the definition of the principal eigenvalue, there exists $u \gg 0$ and $v \gg 0$ such that 
	$(A+B)u=\lambda u$ and $A v =\lambda v$. 
	Choose $t_0>0$ such that $u -t_0 v \geq 0 $ but $u- t_0 v \not \in \mathrm{Int}(K)$. If $u -t_0 v=0$, then
	\begin{equation}\label{equ:A_B}
		A(u-t_0 v)+Bu =\lambda(u-t_0v)
	\end{equation}
	implies that $Bu =0$, and hence, $B$ is zero operator by Lemma \ref{lem:positive:zero}, which is a contradiction.
	If $u -t_0 v>0$, since $A$ is strongly positive, then \eqref{equ:A_B} yields that $u-t_0v= \frac{1}{\lambda}[A(u-t_0 v)+Bu] \gg 0$, which is impossible. 
\end{proof}

In the remaining of the section, we present some results of cooperative and irreducible matrices, which provide some primary tools for the analysis later in this paper. 
\begin{lemma}\label{lem:irreducible}
	If $A$ is an $l \times l$ non-negative and irreducible matrix whose diagonal elements are all positive, then $A^{p}$ is strongly positive for all $p \geq l-1$. 
\end{lemma}

\begin{proof}
	Choose $\epsilon_0$ small such that $A-\epsilon_0 I$ is non-negative.
	Since $A$ can be written as $A-\epsilon_0 I + \epsilon_0 I$, so the conclusion holds for $p= l-1$ by \cite[(8.3.5)]{meyer2000matrix}. An easy computation yields that $A^{p} \bm{x}= A^{l-1} A^{p-l+1} \bm{x} \geq A^{l-1} \epsilon_0^{p-l+1} \bm{x} \gg0$ for all $p>l-1$, $\bm{x}>0$. This completes the proof.
\end{proof}
We recall a result about the spectral bound of an irreducible cooperative matrix (see, e.g., \cite{lancaster1985matrices}). For completeness, here we provide an elementary proof.
\begin{lemma}\label{lem:A+B>A}
	Assume that $A$ is a cooperative and irreducible $l \times l$ matrix, $B$ is a non-negative $l \times l$ non-zero matrix, then $s(A+B) > s(A)$.
\end{lemma}

\begin{proof}
	Without loss of generality, we assume that $A$ is a non-negative matrix whose diagonal elements are all positive, for otherwise we can replace $A$ by $A+\mu I$ with some large positive $\mu$. Here $I$ is the identity $l \times l$ matrix. Lemma \ref{lem:spectral:bound_radius} yields that $s(A)=r(A)$ and $s(A+B)=r(A+B)$. It suffices to show that $r(A+B) >r(A)$.
	
	According to Lemma \ref{lem:irreducible}, $A^{l-1}$ and $A^{l}$ are strongly positive. We notice that $BA^{l-1}$ is nonzero, so is $(A+B)^{l} -A^{l}$. From Lemma \ref{lem:strongly}, it follows that $[r(A+B)]^{l}=r((A+B)^{l}) >r(A^{l})=[r(A)]^{l}$. This completes the proof.
\end{proof}

We now present the last result of this section.
\begin{lemma} \label{lem:A_split}
	Assume that $A=(a_{ij})_{l \times l}$ is a cooperative and irreducible $l \times l$ matrix and $A$ can be split into
	$$
	A=
	\left(
	\begin{array}{cc}
		A_{11}& A_{12} \\ 
		A_{21}& A_{22}
	\end{array} 
	\right).
	$$
	Here $A_{11}$, $A_{12}$, $A_{21}$, $A_{22}$ are $l_1 \times l_1$, $l_1 \times l_2$, $l_2 \times l_1$, $l_2 \times l_2$ matrices. Let $I$, $I_1$ and $I_2$ be $l \times l$, $l_1 \times l_1$ and $l_2 \times l_2$ identity matrices, respectively. Then the following statements are valid:
	\begin{enumerate}
		\item[\rm (i)] $s(A) > s(A_{22}).$
		\item[\rm (ii)] $
		H=s(A_{11} + A_{12}(H I_2-A_{22})^{-1}A_{21})
		$, where $H:=s(A)$.
		\item[\rm (iii)] Write 
		$
		A_{\bm{\mu}}:=A-C,
		$
		where $\bm{\mu}=(\mu_1,\cdots,\mu_{l_1})$ and $C={\rm diag}(\mu_1,\cdots,\mu_{l_1},0,\cdots,0)$ with $\mu_i>0$, $i=1,\cdots,l_1$. We then have 
		$
		\lim\limits_{\min_{1 \leq i \leq l_1} \mu_i \rightarrow + \infty} s(A_{\bm{\mu}}) = s(A_{22}).
		$
		\item[\rm (iv)] $A_{11} + A_{12}(\lambda I_2-A_{22})^{-1}A_{21}$ is irreducible for all $\lambda > s(A_{22})$.
	\end{enumerate}
\end{lemma}

\begin{proof}
	(i)
	Without loss of generality, we assume that $A$ is a non-negative matrix whose diagonal elements are all positive, for otherwise we can replace $A$ by $A+c I$ with some large positive $c$.
	Let	$B={\rm diag}(A_{11},A_{22})	$. It is easy to see that 
	$$
	s(A) \geq s(B)= \max (s(A_{11}),s(A_{22})) \geq s(A_{22}).
	$$ 
	Suppose, by contradiction, that $\lambda:=s(A)=s(A_{22})$.  Lemma \ref{lem:spectral:bound_radius} yields that $s(A)=r(A)$.
	
	According to Lemma \ref{lem:irreducible}, $A^{l-1}$ and $A^{l}$ are strongly positive. By the Perron-Frobenius theorem (see, e.g., \cite[Theorem 4.3.1]{smith2008monotone}), there exist $\bm{y}^2 \in \mathbb{R}_+^{l_2} \setminus \{0 \}$ and $\bm{x}=((\bm{x}^1)^T,(\bm{x}^2)^T)^T \in \mathrm{Int} (\mathbb{R}_+^{l})$ with $\bm{x}^1\in \mathrm{Int} (\mathbb{R}_+^{l_1})$ and $\bm{x}^2\in \mathrm{Int} (\mathbb{R}_+^{l_2})$ such that $A^{q} \bm{x} =\lambda^q \bm{x}$ and $A_{22}^{q} \bm{y}^2 =\lambda^q \bm{y}^2$. Let $\bm{w} =(0,\cdots,0,(\bm{x}^2)^T)^T$. Since $A^{q}$ is strongly positive, it follows that $\lambda^q \bm{x}= A^{q} \bm{x} \gg A^{q} \bm{w}$ and $\lambda^q \bm{x}^2 \gg (A^{q} \bm{w})^2$, where $(A^q\bm{w})^2$ is the last $l_2$ components of $A^q\bm{w}$. In view of $(A^q\bm{w})^2 \geq A_{22}^q \bm{x}^2$, we have
	$$
	\bm{z}^2:= \lambda^q \bm{x}^2 -A_{22}^{q} \bm{x}^2 =\lambda^q \bm{x}^2 - (A^q\bm{w})^2 +(A^q\bm{w})^2 -A_{22}^{q} \bm{x}^2\gg 0.$$ Choose $t_0>0$ such that 
	$\bm{x}^2 - t_0 \bm{y}^2 \in \mathbb{R}_+^{l_2}$ but $\bm{x}^2 - t_0 \bm{y}^2 \not\in \mathrm{Int}(\mathbb{R}_+^{l_2})$. If $\bm{x}^2 - t_0 \bm{y}^2 = 0$, then $\bm{z}^2 =\lambda^q(\bm{x}^2-t_0 \bm{y}^2) - A_{22}^q(\bm{x}^2-t_0 \bm{y}^2)=0$, which contradicts with $\bm{z}^2 \gg 0$. If $\bm{x}^2 - t_0 \bm{y}^2 \neq 0$, then 
	$
	\bm{x}^2 -t_0 \bm{y}^2= \lambda^{-q}[A_{22}^{q} (\bm{x}^2 -t_0 \bm{y}^2) + \bm{z}^2] \gg 0
	$,
	which is impossible. 
	
	(ii)
	We first prove that $H \in \sigma(A_{11} + A_{12}(H I_2-A_{22})^{-1}A_{21})$. Thanks to $H \in \sigma(A)$, it is easy to see ${\det}(H I-A)=0$, where $\det(F)$ is the determinant of $F$. In addition, part (i) implies that $\det(H I_2-A_{22}) \neq 0$.
	Thus, $\det(H I_1 - A_{11} - A_{12}(H I_2-A_{22})^{-1}A_{21})=0$ follows from
	$$
	\det (H I -A) =\det(H I_2-A_{22}) \det(H I_1 - A_{11} - A_{12}(H I_2-A_{22})^{-1}A_{21}).
	$$
	Therefore, $H \in \sigma(A_{11} + A_{12}(H I_2-A_{22})^{-1}A_{21})$
	and $s(A_{11} + A_{12}(H I_2-A_{22})^{-1}A_{21}) \geq H$. 
	
	Suppose that the inequality holds, that is, $\lambda_0:=s(A_{11} + A_{12}(H I_2-A_{22})^{-1}A_{21})>H$. By the Perron-Frobenius theorem (see, e.g., \cite[Theorem 4.3.1]{smith2008monotone}), we have $\lambda_0 \in \sigma(A_{11} + A_{12}(H I_2-A_{22})^{-1}A_{21})$ and
	$
	\det (\lambda_0 I_1 - A_{11} - A_{12}(H I_2-A_{22})^{-1}A_{21})=0.
	$
	Write 
	$$W:=
	\left(
	\begin{matrix}
		A_{11}-\lambda_0 I_1+H I_1& A_{12}\\
		A_{21}& A_{22}
	\end{matrix}
	\right).
	$$
	It then follows that $H \in \sigma(W)$ from
	$$
	\det(H I-W)=\det(H I_2-A_{22})\det(\lambda_0 I_1 - A_{11} - A_{12}(H I_2-A_{22})^{-1}A_{21})=0,
	$$
	Lemma \ref{lem:A+B>A} yields a contradiction that $H \leq s(W) < s(A)=H$, which leads to $H=s(A_{11} + A_{12}(H I_2-A_{22})^{-1}A_{21})$.
	
	(iii) Let $\lambda(\bm{\mu}):=s(A_{\bm{\mu}})$ be the principal eigenvalue of $A_{\bm{\mu}}$ with a strongly positive eigenvector $\bm{x}(\bm{\mu})=((\bm{x}^1)^T(\bm{\mu}),(\bm{x}^2)^T(\bm{\mu}))^T=(x_1(\bm{\mu}),\cdots,x_l(\bm{\mu}))^T$, i.e.,
	\begin{equation}\label{equ:A:mu}
		\begin{cases}
			\sum_{j=1}^{l} a_{ij} x_j(\bm{\mu})  = (\lambda (\bm{\mu})+ \mu_i) x_i(\bm{\mu}), ~ i=1,\cdots,l_1,\\
			A_{21} \bm{x}^1(\bm{\mu}) + A_{22} \bm{x}^2(\bm{\mu})= \lambda (\bm{\mu}) \bm{x}^2(\bm{\mu}).
		\end{cases}
	\end{equation}
	For each $i=1,\cdots,l_1$, we divide the $i$-th equation of \eqref{equ:A:mu} by $\mu_i$ to obtain
	\begin{equation*}
		\begin{cases}
			\frac{1}{\mu_i}\sum_{j=1}^{l} a_{ij} x_j(\bm{\mu})  = (\frac{1}{\mu_i}\lambda (\bm{\mu})+ 1) x_i(\bm{\mu}), ~ i=1,\cdots,l_1,\\
			A_{21} \bm{x}^1(\bm{\mu}) + A_{22} \bm{x}^2(\bm{\mu})= \lambda (\bm{\mu}) \bm{x}^2(\bm{\mu}).
		\end{cases}
	\end{equation*}
	
	Lemma \ref{lem:A+B>A} yields that $\lambda(\bm{\mu})$ is decreasing with respect to $\mu_i \in (0,+\infty)$, $i=1,\cdots,l_1$. Notice that $\lambda(\bm{\mu}) \geq s(A_{22})$ for all $\bm{\mu}$ with $\mu_i >0$, $i=1,\cdots,l_1$, we have 
	$$\lambda_*:=\lim\limits_{\min_{1 \leq i \leq l_1} \mu_i \rightarrow +\infty} \lambda (\bm{\mu}) $$ exists and $\lambda_* \geq s(A_{22})$. We normalize $\bm{x}(\bm{\mu})$ by $\sum_{i=1}^l x_i(\bm{\mu})=1$, so there exist a sequence $\bm{\mu}_n=(\mu_{n,1},\cdots,\mu_{n,l_1})$ with $\min_{1 \leq i \leq l_1} \mu_{n,i} \rightarrow +\infty$ such that $\bm{x} (\bm{\mu}_n)$ converges to some $\bm{x}_*=((\bm{x}_*^1)^T,(\bm{x}_*^2)^T)^T=(x_{*,1},\cdots,x_{*,l})$ as $\min_{1 \leq i \leq l_1} \mu_{n,i} \rightarrow +\infty$ and
	\begin{equation*}
		\begin{cases}
			\bm{x}_*^1=0,\\
			A_{21} \bm{x}_*^1 + A_{22} \bm{x}_*^2= \lambda_* \bm{x}_*^2,\\
			\sum_{i=1}^l x_{*,i}=1.
		\end{cases}
	\end{equation*}
	This implies that $\lambda_*$ is an eigenvalue of $A_{22}$, and hence, $\lambda_* \leq s(A_{22})$. Therefore, $\lambda_*=s(A_{22})$.
	
	(iv)
	For a given $\lambda > s(A_{22})$, write 
	$$M:=
	\left(
	\begin{matrix}
		\lambda I_2-A_{22}& 0\\
		-A_{12}& I_1
	\end{matrix}
	\right),~
	N:=
	\left(
	\begin{matrix}
		0& A_{21}\\
		0& A_{11}+I_1
	\end{matrix}
	\right).
	$$
	Thus,
	$$M^{-1}
	=
	\left(
	\begin{matrix}
		(\lambda I_2-A_{22})^{-1}& 0\\
		A_{12}(\lambda I_2-A_{22})^{-1}& I_1
	\end{matrix}
	\right),~
	M^{-1}N
	=\left(
	\begin{matrix}
		0& (\lambda I_2-A_{22})^{-1}A_{21}\\
		0& A_{12}(\lambda I_2-A_{22})^{-1}A_{21}+A_{11}+I_1
	\end{matrix}
	\right).
	$$
	We notice that $A$ is irreducible, so is $N-M$. Moreover, for each $j=l_1+1,\cdots,l$, the $j$-th column of $N$ is nonzero.
	According to \cite[Lemma 3.4]{schneider1984theorems} and \cite[Theorem 2.2.7]{berman1994nonnegative}, $A_{12}(\lambda I_2-A_{22})^{-1}A_{21}+A_{11}+I_1$, and hence, $A_{12}(\lambda I_2-A_{22})^{-1}A_{21}+A_{11}$ is irreducible.
\end{proof}

\section{Principal eigenvalue for the partially degenerate case} \label{sec:existence}
In this section, we establish the principal eigenvalue theory for partially degenerate nonlocal dispersal systems. Let $X:=C(\overline{\Omega},\mathbb{R}^l)$ be an ordered Banach space with the positive cone $X_+:=C(\overline{\Omega},\mathbb{R}^l_+)$ and the maximum norm
$$
\Vert \bm{u} \Vert_X= \max_{1 \leq i \leq l,x \in \overline{\Omega}} \vert u_{i} (x) \vert,
$$
where $\bm{u} = (u_1, \cdots,u_l)^T \in X$ is an $l$-dimensional vector-valued function. For each $1 \leq i \leq l$, $x \in \overline{\Omega}$, $y \in \overline{\Omega}$, let $K_i(x,y)$ be a non-negative continuous function on $\overline{\Omega} \times \overline{\Omega}$ defined by $K_i(x,y)=d_ik_i(x,y)$.
Recall $A(x)= (a_{ij}(x))_{l \times l}$ is a continuous matrix-valued function of $x \in \overline{\Omega}$ defined by \eqref{equ:A(x)}. Then we have the following observations.
\begin{lemma}\label{lem:AK}
	Assume that {\rm (H1)--(H3)} and {\rm (H4$'$)} hold. Then the following statements are valid:
	\begin{enumerate}
		\item[\rm (i)] For any $x \in \overline{\Omega}$, $A(x)$ is cooperative.
		\item[\rm (ii)] For any $x \in \overline{\Omega}$, $A(x)$ is irreducible.
		\item[\rm (iii)] For any $x \in \overline{\Omega}$ and $1 \leq i \leq l_1$, $K_i(x,x)>0$; for any $x,y \in \overline{\Omega}$ and $l_1+1 \leq i \leq l$, $K_i(x,y)=0$.
	\end{enumerate}
\end{lemma}
Let $\mathcal{A}$ and $\mathcal{K}$ be two linear operators on $X$ defined by:
$$
[\mathcal{A} \bm{u} ](x):=A(x) \bm{u}(x),~ x \in \overline{\Omega}, ~ \bm{u} \in X,
$$
$$
\mathcal{K} \bm{u}:=(\mathcal{K}_1 u_1,\cdots,\mathcal{K}_i u_i,\cdots,\mathcal{K}_l u_l)^T, ~ \bm{u} \in X,
$$
where
$$
[\mathcal{K}_i v](x):
=
\int_{\Omega} K_i(x,y) v(y) \mathrm{d} y, 1 \leq i \leq l,
~ x \in \overline{\Omega},~ v \in C(\overline{\Omega}).
$$
We remark that $[\mathcal{K}_i v](x)=0$ for $l_1+1 \leq i \leq l$, $x \in \overline{\Omega}$, $v \in C(\overline{\Omega})$ due to Lemma \ref{lem:AK}(iii).
Define an operator $\mathcal{Q}: X \rightarrow X$ by
$$
\mathcal{Q}:=\mathcal{A}+\mathcal{K}.
$$
In this section, we consider the following eigenvalue problem 
\begin{equation}\label{equ:eig:cQ}
	\mathcal{Q} \bm{u} = \lambda \bm{u}.
\end{equation}
Recall $H(x)= s(A(x)),~ \forall x \in \overline{\Omega},$ and $\eta= \max_{x \in \overline{\Omega}} H(x)$.
We next review a sufficient condition of the existence of the principal eigenvalue for the non-degenerate case (see, e.g., \cite[Theorem 2.2]{bao2017criteria} or \cite[Theorem 2.1 and Remark 1]{li2017eigenvalue}), which provide a powerful tool in our analysis.
\begin{lemma}\label{lem:non-dene}
	Assume that {\rm (H1)--(H2)} hold and $K_i(x,x)>0$ for all $x \in \overline{\Omega}$ and $1 \leq i \leq l$. If there is an open set $\Omega_0 \subset \Omega$ such that $(\eta- H)^{-1} \not\in L^1 (\Omega_0)$, then $s(\mathcal{Q}) $ is the principal eigenvalue of $\mathcal{Q}$. 
\end{lemma}

Since $\mathcal{K}_i$ is compact on $C(\overline{\Omega})$ for $1 \leq i \leq l_1$, it is easy to see that $\mathcal{K}$ is compact on $X$. Thus, the structure of the essential spectrum of $\mathcal{Q}$ is established by the following lemma (see, e.g., \cite[Appendix B]{liang2017principal}). 
\begin{lemma}\label{lem:eta}
	Assume that {\rm (H1)} hold. Then the following statements are valid:
	\begin{enumerate}[\rm (i)]
		\item $\sigma_e(\mathcal{Q})=\sigma(\mathcal{A})=\cup_{x \in \overline{\Omega}} \sigma(A(x))$.
		\item $s_e(\mathcal{Q})=s(\mathcal{A})=\eta$.
		\item $s_e(\mathcal{Q}+c\mathcal{I})=s(\mathcal{A})+c=\eta+c$, $\forall c\in\mathbb{R}$, where $\mathcal{I}$ is an identity map on $X$.
	\end{enumerate}
\end{lemma}

\begin{lemma}\label{lem:strongly_positive}
	Assume that {\rm (H1)--(H3)} and {\rm (H4$'$)} hold. If $A(x)$ are non-negative matrices whose diagonal elements are positive for all $x\in \overline{\Omega}$, then $\mathcal{Q}$ is eventually strongly positive, that is, there exists some $n_0$ such that $\mathcal{Q}^{n}$ is strongly positive for $n \geq n_0$. Furthermore, $s_e(\mathcal{Q}^{n_0})=[s_e(\mathcal{Q})]^{n_0}$.
\end{lemma}

\begin{proof}
	By the arguments similar to those in \cite[Proposition 2.2]{shen2010spreading}, $\mathcal{K}_1$ is eventually strongly positive on $C(\overline{\Omega})$, that is, there exists some $n_1$ such that $\mathcal{K}_1^{n_1}$ is strongly positive on $C(\overline{\Omega})$. Lemma \ref{lem:irreducible} implies that $A^{l} (x)$ is strongly positive for all $x\in \overline{\Omega}$. 
	For any given $\bm{u}=(u_1,\cdots,u_l) \in X$ with $\bm{u} > 0$, there exists some $x_0 \in \overline{\Omega}$ and  $1\leq i_0 \leq l$ such that $u_{i_0} (x_0) >0$. Write $\bm{v}:=(v_1,\cdots,v_l)=\mathcal{A}^l \bm{u}$ and $\bm{w}:=(w_1,\cdots,w_l)=\mathcal{K}^{n_1} \bm{v}$. We notice that $\bm{v}(x) = A^l(x) \bm{u} (x)$, so
	$v_1(x_0)>0$ and $w_1 (x)=[\mathcal{K}_1^{n_1} v_1](x) >0$ for all $x\in \overline{\Omega}$. By applying the fact that $A^{l} (x)$ is strongly positive for all $x\in \overline{\Omega}$ again, we conclude that
	$$
	\mathcal{Q}^{l+n_1 +l} \bm{u} \geq \mathcal{A}^l \mathcal{K}^{n_1} \mathcal{A}^l \bm{u} =\mathcal{A}^{l} \bm{w} \gg 0,
	$$
	and hence, $\mathcal{Q}$ is eventually strongly positive by letting $n_0:=l+n_1+l$.
	Remark \ref{rem:spectral:bound_radius} and Lemma \ref{lem:eta} yield that $s(\mathcal{A}^{n_0})=r(\mathcal{A}^{n_0})=[r(\mathcal{A})]^{n_0} = [s(\mathcal{A})]^{n_0} = [s_e(\mathcal{Q})]^{n_0} $. Noting that $
	\mathcal{Q}^{n_0}- \mathcal{A}^{n_0}=(\mathcal{K}+\mathcal{A})^{n_0}- \mathcal{A}^{n_0} 
	$ is compact, we  have $\sigma_e(\mathcal{Q}^{n_0})=\sigma_e(\mathcal{A}^{n_0})=\sigma(\mathcal{A}^{n_0})$ due to \cite[Theorem 7.27]{schechter1971principles}. Taking supremum in the real part of the respective spectrum, we obtain $s_e(\mathcal{Q}^{n_0})= s(\mathcal{A}^{n_0})= [s_e(\mathcal{Q})]^{n_0} $.
\end{proof}

\begin{theorem}\label{thm:prop:principal}
	Assume that {\rm (H1)--(H3)} and {\rm (H4$'$)} hold. If there exists some $\bm{u}_* \gg 0$ and $\lambda_* \in \mathbb{R}$ such that 
	$\mathcal{Q} \bm{u}_* = \lambda_* \bm{u}_*$. Then the following statements are valid:
	\begin{enumerate}[\rm (i)]
		\item $\lambda_* > \eta= s_e(\mathcal{Q})$.
		\item $\lambda_*$ is an algebraically simple eigenvalue.
		\item $\bm{u}_*$ is the unique non-negative eigenvector up to scalar multiplications on $X$.
		\item $s(\mathcal{Q})= \lambda_*$ and there exists $\delta > 0$ such that 
		$\mathrm{Re} \lambda < \lambda_* -\delta$ for any $\lambda \in \sigma(\mathcal{Q}) \setminus \{ \lambda_* \}$.
	\end{enumerate}
\end{theorem}
\begin{proof}
	Choose $x_0 \in \overline{\Omega}$ such that $s(A(x_0)) = \eta$.
	In view of $\mathcal{Q} \bm{u}_*= (\mathcal{A} + \mathcal{K} ) \bm{u}_* = \lambda_* \bm{u}_*$, we have 
	$ \lambda_* \bm{u}_*(x_0) = \mathcal{A} \bm{u}_*(x_0) + \mathcal{K} \bm{u}_*(x_0) $, and hence
	$ \lambda_* \bm{u}_*(x_0) > A(x_0) \bm{u}_*(x_0)$.
	Mutiplying by the left principal eigenvector of $A(x_0)$ to the above inequality, we obtain $\lambda_*> \eta$. 
	
	Let $I$ be an $l \times l$ identity matrix and $\mathcal{I}$ be an identity map on $X$, and take $c>0$ such that $A(x)+cI$ are non-negative matrices whose diagonal elements are positive for all $x\in \overline{\Omega}$. In view of Lemma \ref{lem:strongly_positive}, $\mathcal{Q} + c\mathcal{I}$ is eventually strongly positive, that is, there exists some $n_0$ such that $\mathcal{Q}^{n}$ is strongly positive for $n \geq n_0$. Lemmas \ref{lem:eta} and \ref{lem:strongly_positive} yield that
	$$
	(\lambda_*+c)^{n_0}>(\eta +c)^{n_0}=[s_e(\mathcal{Q})+c]^{n_0}=[s_e(\mathcal{Q}+c \mathcal{I})]^{n_0}=s_e((\mathcal{Q}+c\mathcal{I})^{n_0}).
	$$
	The remaining parts can be derived by Lemma \ref{lem:gene:K-R:s}.
\end{proof}

Define $X_1:=C(\overline{\Omega},\mathbb{R}^{l_1})$, $X_2:=C(\overline{\Omega},\mathbb{R}^{l_2})$, $X_{1,+}:=C(\overline{\Omega},\mathbb{R}^{l_1}_+)$ and $X_{2,+}:=C(\overline{\Omega},\mathbb{R}^{l_2}_+)$ with the maximum norm
$$
\Vert \bm{u}^1 \Vert_{X_1}= \max_{1 \leq i \leq l_1,x \in \overline{\Omega}} \vert u_{i} (x) \vert,
\text{ and }
\Vert \bm{u}^2 \Vert_{X_2}= \max_{l_1+1 \leq i \leq l,x \in \overline{\Omega}} \vert u_{i} (x) \vert.
$$
Here $\bm{u}^1 = (u_1, \cdots,u_{l_1})^T \in X_1$ and $\bm{u}^2 = (u_{l_1+1}, \cdots,u_{l_1 + l_2})^T \in X_2$ are $l_1$-dimensional and $l_2$-dimensional vector-valued functions, respectively. Then $(X_1,X_{1,+})$ and $(X_2,X_{2,+})$ are two ordered Banach spaces. Recall that $D_1(x)= {\rm diag}(d_1 \chi_1(x),\cdots,d_{l_1} \chi_{l_1}(x))$ with $d_i>0$, $i=1,\cdots,l_1$. For convenience, for any $x\in \overline{\Omega}$, let
$$
A_{11}(x):=M_{11}(x)-D_1(x),
A_{12}(x):=M_{12}(x),
A_{21}(x):=M_{21}(x),
A_{22}(x):=M_{22}(x).
$$
Define four linear operators 
$\mathcal{A}_{11}: X_1 \rightarrow X_1$, 
$\mathcal{A}_{12}: X_2 \rightarrow X_1$, 
$\mathcal{A}_{21}: X_1 \rightarrow X_2$ and
$\mathcal{A}_{22}: X_2 \rightarrow X_2$ by
$$
\begin{aligned}
	&[\mathcal{A}_{11} \bm{u}^1] (x) := A_{11} (x) \bm{u}^1 (x),\quad 
	&[\mathcal{A}_{12} \bm{u}^2] (x) := A_{12} (x) \bm{u}^2 (x),\\
	&[\mathcal{A}_{21} \bm{u}^1] (x) := A_{21} (x) \bm{u}^1 (x), 
	&[\mathcal{A}_{22} \bm{u}^2] (x) := A_{22} (x) \bm{u}^2 (x).
\end{aligned}
$$
Let $\hat{\mathcal{K}}$ be a linear operator from $ X_1$ to $X_1$ defined by
$$
\hat{\mathcal{K}} \bm{u}^1:=(\mathcal{K}_1 u_1,\mathcal{K}_2 u_2,\cdots,\mathcal{K}_{l_1} u_{l_1})^T.
$$
Thus, \eqref{equ:eig:cQ} can be rewritten as 
\begin{equation}\label{equ:eig:ck}
	\left\{
	\begin{aligned}
		\hat{\mathcal{K}} \bm{u}^1 +&\mathcal{A}_{11} \bm{u}^1 + \mathcal{A}_{12} \bm{u}^2= \lambda \bm{u}^1,\\
		&\mathcal{A}_{21} \bm{u}^1 + \mathcal{A}_{22} \bm{u}^2= \lambda \bm{u}^2.
	\end{aligned}
	\right.
\end{equation}

\begin{lemma}\label{lem:eta22}
	Assume that {\rm (H1)--(H2)} hold. Then the following statements are valid:
	\begin{enumerate}[\rm (i)]
		\item $\eta_{22}=s(\mathcal{A}_{22})=\max_{x \in \overline{\Omega}} s(A_{22}(x))$.
		\item $s(A(x))>s(A_{22}(x))$ for all $x \in \overline{\Omega}$.
		\item $\eta>\eta_{22}$.
	\end{enumerate}
\end{lemma}

\begin{proof}
	In view of $\eta_{22}=\max_{x \in \overline{\Omega}} s(M_{22}(x))$ and $A_{22}(x)=M_{22}(x)$, parts (i) and (ii) follow from \cite[Proposition 2.7]{liang2019principal} and Lemma \ref{lem:A_split}(i), respectively. By (i), (ii) and the definition of $\eta$, we obtain (iii).
\end{proof}

For any $x \in \overline{\Omega}$, write 
$$
F_{\lambda}(x):= A_{11}(x) + A_{12}(x) (\lambda I_2 - A_{22} (x))^{-1} A_{21}(x),~ \lambda > s(A_{22}(x)).
$$
Here $I_2$ is the identity $l_2 \times l_2$ matrix, and $F_{\eta}(x)$ is well-defined for all $x\in \overline{\Omega}$ due to $\eta> \eta_{22}$. Let
\begin{equation}\label{equ:h}
	h(x): = s(F_{\eta}(x)),~ x \in \overline{\Omega}.
\end{equation}
For any $\lambda > \eta_{22}$, let 
$\mathcal{F}_{\lambda}$ and $\mathcal{T}_{\lambda}$ be two families of linear operators from $X_1$ to $X_1$ defined by
$$
\mathcal{F}_{\lambda}:= \mathcal{A}_{12}(\lambda \mathcal{I}_2- \mathcal{A}_{22})^{-1} \mathcal{A}_{21} + \mathcal{A}_{11},\quad
\mathcal{T}_{\lambda}:= \mathcal{F}_{\lambda} + \hat{\mathcal{K}}.
$$
Here $\mathcal{I}_2$ is an identity map on $X_2$.
Therefore, \eqref{equ:eig:ck} can be further written as the following generalized eigenvalue problem
\begin{equation}\label{equ:eig:L_lambda}
	\mathcal{T}_{\lambda} \bm{u}^1 = \lambda \bm{u}^1.
\end{equation}
\begin{lemma}\label{lem:L_F}
	Assume that {\rm (H1)--(H3)} and {\rm (H4$'$)} hold. Then the following statements are valid:
	\begin{enumerate}[\rm (i)]
		\item $\sigma_e(\mathcal{T}_{\lambda})=\sigma_e(\mathcal{F}_{\lambda})$, $\lambda > \eta_{22}$.
		\item $\sigma(\mathcal{F}_{\eta})=\sigma_e(\mathcal{F}_{\eta})= \cup_{x \in \overline{\Omega}} \sigma(F_{\eta}(x))$.
		\item $s(F_{H(x)}(x)) = H(x)$, $\forall x \in \overline{\Omega}$.
		\item $h(x) \leq H(x)$, $\forall x \in \overline{\Omega}$.
		\item $s(\mathcal{F}_{\eta}) = \max_{x \in \overline{\Omega}} h(x)=\eta$.
		\item $s(\mathcal{F}_{\lambda})$ is continuous and non-increasing with respect to $\lambda \in ( \eta_{22},+\infty)$.
		\item $s(\mathcal{T}_{\lambda})$ is continuous and non-increasing with respect to $\lambda \in ( \eta_{22},+\infty)$. 
	\end{enumerate} 
\end{lemma}

\begin{proof}
	According to \cite[Theorem 7.27]{schechter1971principles} and \cite[Proposition 2.7]{liang2019principal}, we obtain (i) and (ii). Part (iii) follows from Lemma \ref{lem:A_split}(ii).
	
	(iv) By \cite[Theorem 1.1]{burlando1991monotonicity}, for any given $x\in \overline{\Omega}$, $s(F_{\lambda}(x))$ is non-increasing with respect to $\lambda > \eta_{22}$. In view of $ \max_{x \in \overline{\Omega}} H(x)=\eta$, we obtain
	$$
	h(x)
	=s(F_{\eta}(x)) 
	\leq s(F_{H(x)}(x)) 
	= H(x),~ \forall x \in \overline{\Omega}.
	$$
	
	(v) It is easy to see that $s(\mathcal{F}_{\eta}) = \max_{x \in \overline{\Omega}} h(x)$ due to (ii). 
	Choose $x_0 \in \overline{\Omega}$ such that $H(x_0)=\max_{x \in \overline{\Omega}} H(x) =\eta$. We then have  
	$$
	h(x_0)=s(F_{\eta}(x_0))=s(F_{H(x_0)}(x_0))=H(x_0)=\eta.
	$$
	This implies that 
	$
	\max_{x \in \overline{\Omega}} h(x) \geq \eta=\max_{x \in \overline{\Omega}} H(x)
	$,
	and hence,
	$$
	\max_{x \in \overline{\Omega}} h(x) = \eta=\max_{x \in \overline{\Omega}} H(x).$$
	
	(vi) For any $\lambda > \eta_{22}$, define 
	$
	p(x,\lambda) :=s(F_{\lambda}(x)),~ x \in \overline{\Omega}
	$,
	and 
	$
	P(\lambda) :=\max_{x \in \overline{\Omega}} p(x,\lambda)
	$.
	According to \cite[Proposition 2.7]{liang2019principal}, we have $$s(\mathcal{F}_{\lambda}) = \max_{x \in \overline{\Omega}} s(F_{\lambda}(x))=P(\lambda), ~\lambda >\eta_{22}.$$
	Since $p$ is jointly continuous in $(x,\lambda) \in \overline{\Omega} \times (\eta_{22},+\infty)$ and non-increasing with respect to $\lambda\in (\eta_{22},+\infty)$ for any $x \in \overline{\Omega}$, it then follows that $P$ is continuous and non-increasing in $\lambda \in (\eta_{22},+\infty)$. 
	
	(vii)
	In view of \cite[Theorem 1.1]{burlando1991monotonicity}, $s(\mathcal{T}_{\lambda})$ is non-increasing with respect to $\lambda \in (\eta_{22},+\infty)$ and $s(\mathcal{T}_{\lambda}) \geq s(\mathcal{F}_{\lambda})$ for all $\lambda > \eta_{22}$. Moreover, Theorem \ref{thm:conti:KA} yields that $s(\mathcal{T}_{\lambda})$ is continuous with respect to $\lambda \in (\eta_{22},+\infty)$. For reader's convenience, we provide a short proof.
	For any given $\lambda_1 > \eta_{22}$, it suffices to show that $s(\mathcal{T}_{\lambda})$ is continuous at $\lambda= \lambda_1$. We proceed according to two cases:
	
	{\it Case 1:} $s(\mathcal{T}_{\lambda_1})> s(\mathcal{F}_{\lambda_1})$.
	
	Thanks to (i) and Lemma \ref{lem:gene:K-R:s}, $s(\mathcal{T}_{\lambda_1})$ is an isolated eigenvalue of $\mathcal{T}_{\lambda_1}$. By the perturbation theory of isolated eigenvalues (see, e.g., \cite[Section IV.3.5]{kato1976perturbation}), $s(\mathcal{T}_{\lambda})$ is continuous at $\lambda=\lambda_1$.
	
	{\it Case 2:} $s(\mathcal{T}_{\lambda_1})= s(\mathcal{F}_{\lambda_1})$.
	
	For any given $\epsilon>0$, there exists $\delta_1>0$ small enough with $\lambda_1> \delta_1 +\eta_{22}$ such that
	$$
	\vert s(\mathcal{F}_{\lambda_1}) - s(\mathcal{F}_{\lambda}) \vert \leq \epsilon,~\forall \lambda \in \mathbb{R} \text{ with } \vert \lambda -\lambda_1\vert \leq \delta_1
	$$
	due to (vi).
	According to \cite[Theorem IV.3.1 and Remark IV.3.2]{kato1976perturbation}, there exists $\delta_2>0$ small enough with $\lambda_1> \delta_2 +\eta_{22}$ such that
	$$
	s(\mathcal{T}_{\lambda}) -s(\mathcal{T}_{\lambda_1}) \leq \epsilon,~\forall \lambda \in \mathbb{R} \text{ with } \vert \lambda -\lambda_1\vert \leq \delta_2.
	$$
	It then follows that
	$$
	s(\mathcal{T}_{\lambda}) -s(\mathcal{T}_{\lambda_1}) \leq \epsilon
	\text{ and }
	-\epsilon+ s(\mathcal{T}_{\lambda_1})
	= -\epsilon +s(\mathcal{F}_{\lambda_1})
	\leq s(\mathcal{F}_{\lambda}) \leq s(\mathcal{T}_{\lambda}),
	$$
	if $\vert \lambda -\lambda_1\vert \leq \delta$, where $\delta=\min(\delta_1,\delta_2)$. Combining the above two inequalities, we conclude that
	$$
	\vert s(\mathcal{T}_{\lambda_1}) - s(\mathcal{T}_{\lambda}) \vert \leq \epsilon,~\forall \lambda \in \mathbb{R} \text{ with } \vert \lambda -\lambda_1\vert \leq \delta.
	$$
	This completes the proof.
\end{proof}

Next, we show the relation between \eqref{equ:eig:cQ} and \eqref{equ:eig:L_lambda}.
\begin{proposition}\label{prop:equivalent:condition}
	Assume that {\rm (H1)--(H3)} and {\rm (H4$'$)} hold. The following statements are equivalent.
	\begin{enumerate}
		\item[\rm (i)] There exists $\lambda_* > \eta$ and $\bm{u}_* \gg 0$ such that $\mathcal{Q} \bm{u}_* = \lambda_* \bm{u}_*$.
		\item[\rm (ii)] $s(\mathcal{Q}) > \eta$.
		\item[\rm (iii)] $s(\mathcal{T}_{\eta}) > s(\mathcal{F}_{\eta})$.
		\item[\rm (iv)] There exists $\lambda_* > \eta $ such that $s(\mathcal{T}_{\lambda_*})=\lambda_*> s(\mathcal{F}_{\lambda_*})$. 
		\item[\rm (v)] There exists a pair of $\lambda_* >\eta $ and $\bm{u}_*^1 > 0$ such that $\mathcal{T}_{\lambda_*} \bm{u}_*^1= \lambda_* \bm{u}_*^1$.
		\item[\rm (vi)] There exists a pair of $\lambda_* >\eta $ and $\bm{u}_*^1 \gg 0$ such that $\mathcal{T}_{\lambda_*} \bm{u}_*^1= \lambda_* \bm{u}_*^1$.
	\end{enumerate}
\end{proposition}

\begin{proof}
	(i) $\Rightarrow$ (ii). It has been proved in Theorem \ref{thm:prop:principal}.
	
	(ii) $\Rightarrow$ (i). Lemma \ref{lem:gene:K-R:s} implies that there exists $\bm{u}_* \gg 0$ such that $\mathcal{Q} \bm{u}_* =\lambda_* \bm{u}_*$.
	
	(iii) $\Rightarrow$ (iv). Define $f(\lambda) :=s(\mathcal{T}_{\lambda}) - \lambda$, $\forall \lambda >\eta_{22}$. By Lemma \ref{lem:L_F}(vii), it is easy to see that $f(\lambda)$ is decreasing and continuous with respect to $\lambda$.
	In view of $s(\mathcal{T}_{\lambda}) > s(\mathcal{A}_{11})$, $\forall \lambda >\eta_{22}$, we then have $f(\lambda) =s(\mathcal{T}_{\lambda}) -\lambda \rightarrow -\infty$ as $\lambda \rightarrow +\infty$. In addition, $f(\eta)= s(\mathcal{T}_{\eta})-\eta=s(\mathcal{T}_{\eta})-s(\mathcal{F}_{\eta})>0$. By the intermediate value theorem, there exists $\lambda_* > \eta$ such that 
	$f(\lambda_*)= s(\mathcal{T}_{\lambda_*}) - \lambda_*=0$. Thus,
	$
	s(\mathcal{T}_{\lambda_*}) =  \lambda_* > \eta = s(\mathcal{F}_{\eta}) \geq s(\mathcal{F}_{\lambda_*})
	$
	follows from Lemma \ref{lem:L_F}(v) and (vi).
	
	(iv) $\Rightarrow$ (iii). We notice that $\lambda_* > \eta$, $f(\lambda_*)= s(\mathcal{T}_{\lambda_*}) -\lambda_*=0$ and $f(\lambda)$ is decreasing with respect to $\lambda$, so $f(\eta) = s(\mathcal{T}_{\eta}) -\eta > 0$, and hence $s(\mathcal{T}_{\eta}) > \eta = s(\mathcal{F}_{\eta})$.
	
	(iv) $\Rightarrow$ (v).
	This can be derived by Lemma \ref{lem:gene:K-R:s} directly.
	
	(v) $\Rightarrow$ (i)(vi). Let $\bm{u}_*^2:=(\lambda_* \mathcal{I}_2 - \mathcal{A}_{22})^{-1} \mathcal{A}_{21} \bm{u}_*^1$ and $\bm{u}_*:=((\bm{u}_*^1)^T,(\bm{u}_*^2)^T)^T>0$. It is easy to see that $\mathcal{Q} \bm{u}_* =\lambda_* \bm{u}_*$. It suffices to show that $\bm{u}_* \gg 0$ and $\bm{u}_*^1 \gg 0$. Choose $c>0$ such that $\lambda_*+c>0$ and for each $x \in \overline{\Omega}$, $A(x) + c I$ is a non-negative matrix whose diagonal elements are positive. Lemma \ref{lem:strongly_positive} yields that $\mathcal{Q} + c \mathcal{I}$ is eventually strongly positive, that is, there exists some $n_0$ such that $(\mathcal{Q} + c I)^{n}$ is strongly positive for $n \geq n_0$. Thus, $\bm{u}_* \gg 0$ and $\bm{u}_*^1 \gg 0$ follow from
	$
	(\mathcal{Q} + c \mathcal{I})^{n_0}\bm{u}_*=(\lambda_*+c)^{n_0} \bm{u}_*
	$.
	
	(vi) $\Rightarrow$ (v). Obviously.
	
	(i) $\Rightarrow$ (v). We use $\bm{u}_*^1$ to denote the first $l_1$ components of $\bm{u}_*$. An easy computation yields that $\mathcal{T}_{\lambda_*} \bm{u}_*^1= \lambda_* \bm{u}_*^1$.
	
	(vi) $\Rightarrow$ (iv). It is a straightforward result of Theorem \ref{thm:prop:principal}.
\end{proof}

In the remaining of this section, we give a sufficient condition for the existence of the principal eigenvalue, that is, provide a criterion to derive one of the equivalent statements in Proposition \ref{prop:equivalent:condition}. We need the following lemma, which is a straightforward result of Lemma \ref{lem:non-dene}.

\begin{lemma}\label{lem:cL_eta:cF_eta}
	Assume that {\rm (H1)--(H3)} and {\rm (H4$'$)} hold. If there is an open set $\Omega_0 \subset \Omega$ such that $(\eta - h)^{-1} \not\in L^1(\Omega_0)$, then $s(\mathcal{T}_{\eta})>s(\mathcal{F}_{\eta})$.
\end{lemma}

The subsequent lemma presents a sufficient condition for the existence of the principal eigenvalue when $l_1=1$.
\begin{lemma}\label{lem:H_h_eta}
	Assume that $l_1=1$ and {\rm (H1)--(H3)} and {\rm (H4$'$)} hold. Then there exists $C>0$ independent of $x$ such that   $$
	\vert h(x) -\eta \vert \leq C \vert H(x) -\eta \vert,~x\in\overline{\Omega}.
	$$
	Moreover, if  $(\eta - H)^{-1} \not\in L^1(\Omega_0)$ for some open set $\Omega_0 \subset \Omega$, then $(\eta - h)^{-1} \not\in L^1(\Omega_0)$.
\end{lemma}

\begin{proof}  
	We notice that $F_{\eta} (x) = h(x)$ and $F_{H(x)}(x) =H(x)$ when $l_1=1$, so
	$$
	\begin{aligned}
		h(x)-H(x)
		&= F_{\eta}(x) - F_{H(x)} (x) \\
		&= (H(x) -\eta)M_{12}(x)(\eta I_2 - M_{22} (x))^{-1} (H(x) I_2-M_{22})^{-1} M_{21} (x),~x\in \overline{\Omega}.
	\end{aligned}
	$$
	Let $C_1 := \max_{x \in \overline{\Omega}} \Vert M_{12}(x) \Vert \Vert (\eta I_2 - M_{12}(x))^{-1} \Vert \Vert (H(x) I_2- M_{12}(x))^{-1} \Vert \Vert M_{21}(x) \Vert $.
	Therefore,
	$$
	\vert h(x) -H(x) \vert \leq C_1 \vert H(x) -\eta \vert,~x\in\overline{\Omega}.
	$$
	We conclude that
	$
	\vert h(x) - \eta \vert \leq \vert h(x) -H(x) \vert + \vert H(x) - \eta \vert \leq (C_1+1) \vert H(x) - \eta \vert,~x\in\overline{\Omega}.
	$
\end{proof}

\begin{remark}\label{rem:k1=1}
	Assume that $l_1=1$, {\rm (H1)--(H3)} and {\rm (H4$'$)} hold, and there exists an open set $\Omega_0 \subset \Omega$ such that $(\eta - H)^{-1} \not\in L^1(\Omega_0)$. According to Proposition \ref{prop:equivalent:condition} and Lemmas \ref{lem:cL_eta:cF_eta} and \ref{lem:H_h_eta}, we have $s(\mathcal{T}_{\eta})>s(\mathcal{F}_{\eta})$, and hence, $s(\mathcal{Q}) >\eta$.
\end{remark}

We are now in a position to prove the main result of this section, that is, to provide a sufficient condition for the existence of the principal eigenvalue.
\begin{theorem}\label{thm:dene}
	Assume that {\rm (H1)--(H3)} and {\rm (H4$'$)} hold.  If there is an open set $\Omega_0 \subset \Omega$ such that $(\eta- H)^{-1} \not\in L^1 (\Omega_0)$, then $s(\mathcal{Q}) $ is the principal eigenvalue of $\mathcal{Q}$. 
\end{theorem}

\begin{proof}
	Define $\tilde{\mathcal{K}}: X \rightarrow X$ by:
	$$
	[\tilde{\mathcal{K}} \bm{u}]:=(\tilde{\mathcal{K}}_1 u_1,\tilde{\mathcal{K}}_2 u_2,\cdots,\tilde{\mathcal{K}}_i u_i,\cdots,\tilde{\mathcal{K}}_l u_l)^T, ~ \bm{u} \in X,
	$$
	where
	$$
	[\tilde{\mathcal{K}}_i v](x):
	=
	\begin{cases}
		\int_{\Omega} K_1(x,y) v(y) \mathrm{d} y, &i=1,\\
		0,&2 \leq i \leq l,
	\end{cases}
	~ x \in \overline{\Omega},~ v \in C(\overline{\Omega}).
	$$
	Define $\tilde{\mathcal{Q}}:= \tilde{\mathcal{K}} +\mathcal{A}$.
	According to Remark \ref{rem:k1=1}, we have $s(\tilde{\mathcal{Q}})=s(\tilde{\mathcal{K}} +\mathcal{A}) >\eta$.
	By \cite[Theorem 1.1]{burlando1991monotonicity}, it follows that 
	$$
	s(\mathcal{Q})=s(\mathcal{K}+\mathcal{A}) \geq s(\tilde{\mathcal{K}} +\mathcal{A})= s(\tilde{\mathcal{Q}})>\eta.
	$$
	This completes the proof.
\end{proof}

\begin{remark}
	The conclusions of Theorem \ref{thm:prop:principal} and \ref{thm:dene} and Proposition \ref{prop:equivalent:condition} are still valid when assumptions {\rm (H1)--(H3)} and {\rm (H4$'$)} are replaced by {\rm (i)--(iii)} of Lemma \ref{lem:AK}.
\end{remark}

\section{Asymptotic behavior}\label{sec:asy}
In this section, we investigate the asymptotic behavior of the spectral bound as the diffusion coefficients go to zero and infinity for the non-degenerate case and partially degenerate case.
We use the same notations $X$, $X_+$, $X_1$, $X_{1,+}$, $X_2$ and $X_{2,+}$ as in section \ref{sec:existence}. Let $\mathcal{M}$ be a bounded linear operator on $X$ defined by
$$
[\mathcal{M} \bm{u} ](x):=M(x) \bm{u}(x),~ x \in \overline{\Omega}, ~ \bm{u} \in X.
$$
For any given $\bm{d}:=(d_1,\cdots,d_l)^T \in \mathbb{R}^l_+$, let $\mathcal{D}(\bm{d})$ be a family of linear operators on $X$ defined by
$$
\mathcal{D}(\bm{d}) \bm{u}:=	
\left([\mathcal{D}(\bm{d})]_1 u_1,[\mathcal{D}(\bm{d})]_2 u_2,\cdots,[\mathcal{D}(\bm{d})]_l u_l \right)^T,~ \bm{u} \in X,
$$
where
$$
\{ [\mathcal{D}(\bm{d})]_i v \}(x)=
d_i\left[ \int_{\Omega} k_i(x,y) v(y) \mathrm{d} y - \chi_i(x) v(x) \right],~1 \leq i \leq l,~ x\in \overline{\Omega},~v \in C(\overline{\Omega}). 
$$
Thus,
$
\mathcal{P}(\bm{d})= \mathcal{D}(\bm{d}) + \mathcal{M}.
$
Recall that $\kappa=\max_{x \in \overline{\Omega}}s(M(x))$.
In view of Lemma \ref{lem:eta}, $\kappa= s(\mathcal{M})$. The following theorem is straightforward consequence of the continuity of the spectral bound $s(\mathcal{P}(\bm{d}))$ with respect to $\bm{d} \in \mathbb{R}^l_+$, which is presented in Theorem \ref{thm:conti:KA}.

\begin{theorem}\label{thm:0}
	Assume that {\rm (H1)} holds. Then
	$$
	\lim\limits_{\max_{1\leq i \leq l} d_{i} \rightarrow 0} s(\mathcal{P}(\bm{d})) =\kappa.
	$$
\end{theorem}

\subsection{Non-degenerate case with large diffusion}
In this subsection, we investigate the asymptotic behavior of the spectral bound of $\mathcal{P}(\bm{d})$ as the diffusion coefficients go to infinity for the non-degenerate case. To this end, we need a series of lemmas.
\begin{lemma}\label{lem:p}
	Assume that {\rm (H3)} holds. Then for each $1 \leq i \leq l$, there exists a unique strongly positive continuous function $p_i$ on $\overline{\Omega}$ with $\int_{\Omega} p_i(x) \mathrm{d} x=1$ such that 
	\begin{equation}\label{equ:pi}
		\int_{\Omega} k_i(x,y) p_i(y) \mathrm{d} y= \chi_i(x)p_i(x), ~\forall x \in \overline{\Omega},
	\end{equation} 
	that is,
	\begin{equation}\label{equ:pi:2}
		\int_{\Omega} k_i(x,y) \chi_i^{-1}(x) p_i(y) \mathrm{d} y= p_i(x),~\forall x \in \overline{\Omega}.
	\end{equation} 
\end{lemma}

\begin{proof}
	For any given $1\leq i \leq l$, by the Krein-Rutman theorem (see, e.g., \cite[Theorem 19.3]{deimling1985nonlinear}), there exists a positive number $r$ and a strongly positive continuous function $p$ on $\overline{\Omega}$ with $\int_{\Omega} p(x) \mathrm{d} x=1$ such that 
	$$
	\int_{\Omega} k_i(x,y) \chi_i^{-1}(x) p(y) \mathrm{d} y= r p(x), ~\forall x \in \overline{\Omega}.
	$$
	This implies that
	$$
	\int_{\Omega} k_i(x,y) p(y) \mathrm{d} y= r \chi_i(x) p(x), ~\forall x \in \overline{\Omega}.
	$$
	Integrating the above equation over $\Omega$, we obtain $r=1$. By Lemma \ref{lem:eventually}(iii), there is no other positive continuous function satisfies \eqref{equ:pi:2}.
\end{proof}
In the rest of the section, we use the notation $p_i$ as in \eqref{equ:pi}.
\begin{lemma}\label{lem:s:bounded}
	Assume that {\rm (H1)--(H4)} hold. We have the estimate $\vert s(P(\bm{d})) \vert \leq C $, where the constant $C$ is independent of $\bm{d}$.
\end{lemma}
\begin{proof}
	Choose $\bar{r}>0$ large enough such that 
	$$\overline{M}_{ij}(x):= \bar{r} p_i(x)p_j^{-1}(x) \geq m_{ij}(x), ~\forall x \in \overline{\Omega},~ 1 \leq i,j \leq l.$$
	Let $\overline{M}=(\overline{M}_{ij})_{l \times l}$ and 
	$\underline{M}=(\underline{M}_{ij})_{l \times l}$ with
	$$
	\underline{M}_{ij}=
	\begin{cases}
		\underline{m},& 1 \leq i=j \leq l,\\
		0,&~ 1 \leq i\neq j \leq l,
	\end{cases}
	$$
	where $\underline{m} = \min_{1\leq i,j \leq l, x \in \overline{\Omega}} m_{ij}(x)$.
	Let 
	$\overline{\mathcal{M}}:~X \rightarrow X$, 
	and $\underline{\mathcal{M}}: X \rightarrow X$ be two bounded linear operators defined by
	$$
	[\overline{\mathcal{M}} \bm{u} ](x):=\overline{M}(x) \bm{u}(x),~
	[\underline{\mathcal{M}} \bm{u} ](x):=\underline{M}(x) \bm{u}(x),~ x \in \overline{\Omega}, ~ \bm{u} \in X.
	$$
	Define
	$$
	\overline{\mathcal{P}}(\bm{d}):= \mathcal{D}(\bm{d}) + \overline{\mathcal{M}} \text{ and }
	\underline{\mathcal{P}}(\bm{d}):= \mathcal{D}(\bm{d}) + \underline{\mathcal{M}}.
	$$
	According to \cite[Theorem 1.1]{burlando1991monotonicity}, it is easy to see that 
	$$
	s(\underline{\mathcal{P}}(\bm{d})) \leq s(\mathcal{P}(\bm{d})) \leq s(\overline{\mathcal{P}}(\bm{d})).
	$$
	It suffices to prove that $s(\underline{\mathcal{P}}(\bm{d}))\geq \underline{m}$ and $s(\overline{\mathcal{P}}(\bm{d}))=\bar{r}l$ for all $ \bm{d}$ with $d_i\geq 0$. Choose $\bm{p}:=(p_1,\cdots,p_l)$. It is easy to see that $\underline{\mathcal{P}}(\bm{d}) \bm{p}=\underline{m}\bm{p}$, and hence, $s(\underline{\mathcal{P}}(\bm{d}))\geq\underline{m}$. Moreover, $s(\overline{\mathcal{P}}(\bm{d}))=\bar{r}l$ follows from $\overline{\mathcal{P}}(\bm{d}) \bm{p}= \bar{r}l \bm{p}$ and Lemma \ref{lem:eventually}(iii).
\end{proof}
In view of (H3), $\chi_i(x)=\int_{\Omega} k_i(y,x) \mathrm{d} y>0$, $\forall x \in \overline{\Omega}$. It then follows that $\underline{\chi}= \min_{1 \leq i \leq l} \min_{x \in \overline{\Omega}} \chi_i(x)>0$.
The following results show that the principal eigenvalue exists for large diffusion coefficients, which is already given in \cite{bao2017criteria}.
\begin{lemma}\label{lem:D:large:non-degen}
	Assume that {\rm (H1)--(H4)} hold. If $\min_{1\leq i \leq l} d_i \rightarrow +\infty$, then $s_e(\mathcal{P}(\bm{d})) \rightarrow -\infty$. Moreover, there exists $\hat{d}>0$ such that  $s(\mathcal{P}(\bm{d}))$ is the principal eigenvalue of $\mathcal{P}(\bm{d})$ whenever $\min_{1\leq i \leq l} d_i \geq \hat{d} $. 
\end{lemma}

\begin{proof}
	By the analysis in Lemmas \ref{lem:eta} and \ref{lem:eta22}, it is not hard to see that $s_e(\mathcal{P}(\bm{d})) \rightarrow -\infty$ as $\min_{1\leq i \leq l} d_i \rightarrow +\infty$. Then the remaining conclusion follows from Lemmas \ref{lem:gene:K-R:s} and \ref{lem:s:bounded}. 
\end{proof}

Recall $\tilde{M}=(\tilde{m}_{ij})_{l \times l}$ and $ \tilde{\kappa} = s(\tilde{M})$, where $\tilde{m}_{ij}=\int_{\Omega} m_{ij} (x) p_j(x) \mathrm{d} x,~ 1 \leq i,j \leq l$.
Now we present the main result of this subsection.
\begin{theorem}\label{thm:D:inf:non-degen}
	Assume that {\rm (H1)--(H4)} hold. If $\min_{1\leq i \leq l} d_{i} \rightarrow +\infty$, then 
	$$s(\mathcal{P}(\bm{d})) \rightarrow \tilde{\kappa}.$$
\end{theorem}
\begin{proof}
	By Lemma \ref{lem:D:large:non-degen}, there exists $\hat{d}>0$ such that $\mathcal{P}(\bm{d})$ has the principal eigenvalue whenever $\min_{1\leq i \leq l} d_i \geq \hat{d} $.
	It is sufficient to prove that
	for any sequence $\bm{d}_n=((d_1)_n,\cdots,(d_l)_n)$, there is a subsequence $\bm{d}_{n_k}$ such that if $\min_{1\leq i \leq l} (d_{i})_{n_k} \rightarrow +\infty$, then $s(\mathcal{P}(\bm{d}_{n_k})) \rightarrow \tilde{\kappa}$. Without loss of generality, we assume that $\min_{1\leq i \leq l} (d_i)_n \geq \hat{d} $ for all $n \geq 1$.
	Let $\bm{u}_n=((u_1)_n,\cdots,(u_l)_n)^T$ be the positive eigenvector of $\mathcal{P}(\bm{d}_n)$ corresponding to $s(\mathcal{P}(\bm{d}_n))$ with normalizing $\Vert \bm{u}_n \Vert_X=1$, that is, 
	$\max_{1\leq i \leq l, x\in \overline{\Omega}} (u_i)_n(x) =1$.
	Thus, for all $1 \leq i \leq l$, $x\in \overline{\Omega}$,
	\begin{equation}\label{equ:diff:eig:non-degen}
		(d_i)_n\left[ \int_{\Omega} k_i(x,y) (u_i)_n(y) \mathrm{d} y -  \chi_i(x) (u_i)_n(x) \right] +\sum_{j=1}^l m_{ij}(x) (u_j)_n(x)= s(\mathcal{P}(\bm{d}_n)) (u_i)_n(x).
	\end{equation}
	We divide the $i$-th equation of \eqref{equ:diff:eig:non-degen} by $(d_i)_n$ to obtain
	$$
	\left[ \int_{\Omega} k_i(x,y) (u_i)_n(y) \mathrm{d} y - \chi_i(x) (u_i)_n(x) \right] +\frac{1}{(d_i)_n}
	\left(
	\sum_{j=1}^l  m_{ij}(x) (u_j)_n(x)- s(\mathcal{P}(\bm{d}_n)) (u_i)_n(x) 
	\right)=0.
	$$
	Noting that Lemma \ref{lem:s:bounded} leads to
	$$
	\frac{1}{(d_i)_n}
	\left(
	\sum_{j=1}^l  m_{ij}(x) (u_j)_n(x)- s(\mathcal{P}(\bm{d}_n)) (u_i)_n(x) 
	\right)
	\rightarrow 0 \text{ as }(d_i)_n \rightarrow +\infty,~\forall 1 \leq i \leq l,
	$$
	we then have
	\begin{equation}\label{equ:inf:1}
		\lim\limits_{n \rightarrow +\infty} 
		\left\vert \chi_i(x)(u_i)_{n}(x) - \int_{\Omega} k_i(x,y) (u_i)_n (y) \mathrm{d} y\right\vert =0,~\forall 1 \leq i \leq l
	\end{equation}
	uniformly on $\overline{\Omega}$. Define 
	$$
	\tilde{k}_i(x,y):=k_i(x,y) \chi_i^{-1}(x), ~\forall i=1,\cdots,l,~ x,y \in \overline{\Omega}.
	$$
	Then \eqref{equ:inf:1} is equivalent to
	\begin{equation}\label{equ:inf:1:e}
		\lim\limits_{n \rightarrow +\infty} 
		\left\vert (u_i)_{n}(x) - \int_{\Omega} \tilde{k}_i(x,y) (u_i)_n (y) \mathrm{d} y\right\vert =0,~\forall 1 \leq i \leq l
	\end{equation}
	uniformly on $\overline{\Omega}$.
	In view of Lemma \ref{lem:s:bounded} and $\Vert \bm{u}_n \Vert_X=1$,
	there exists a subsequence $n_k \rightarrow +\infty$ such that $s(\mathcal{P}(\bm{d}_{n_k}))$ converges to some $\bar{\lambda}$ and $(u_i)_{n_k}$ weakly converges to some $w_i$ in $L^2(\Omega)$ as $n_k \rightarrow +\infty$ for all $1\leq i\leq l$. This implies that
	$$
	\lim\limits_{n_k \rightarrow +\infty}
	\left\vert
	\int_{\Omega} \tilde{k}_i(x,y) (u_i)_{n_k}(y) \mathrm{d} y - \int_{\Omega} \tilde{k}_i(x,y) w_i (y) \mathrm{d} y \right\vert=0,~\forall 1 \leq i \leq l
	$$
	pointwise on $\overline{\Omega}$, and hence,
	$$
	\lim\limits_{n_k \rightarrow +\infty} 
	\left\vert(u_i)_{n_k}(x) - \int_{\Omega} \tilde{k}_i(x,y) w_i(y) \mathrm{d} y\right\vert =0,~\forall 1 \leq i \leq l,
	$$
	pointwise on $\overline{\Omega}$.
	Notice that $u_i$ is bounded and 
	$$
	\int_{\Omega} \tilde{k}_i(x,y) w_i (y) \mathrm{d} y 
	\leq \left[\int_{\Omega} \tilde{k}_i^2(x,y)  \mathrm{d} y\right]^{\frac{1}{2}} 
	\left[\int_{\Omega} w_i^2 (y) \mathrm{d} y\right]^{\frac{1}{2}}, ~\forall x \in \overline{\Omega},
	$$
	which is bounded.
	By the dominated convergence theorem, for any $1 \leq i \leq l$, we have
	$$
	\lim\limits_{n_k \rightarrow +\infty}
	\int_{\Omega} \left\vert (u_i)_{n_k}(x) - \int_{\Omega} \tilde{k}_i(x,y) w_i (y) \mathrm{d} y\right\vert \mathrm{d} x=0.
	$$
	Since $\tilde{k}_i(x,y)$ is bounded, it then follows that
	\begin{equation}\label{equ:inf:2}
		\lim\limits_{n_k \rightarrow +\infty}
		\left\vert \int_{\Omega} \tilde{k}_i(x,y) (u_i)_{n_k}(y) \mathrm{d} y - \int_{\Omega} \tilde{k}_i(x,y) \int_{\Omega} \tilde{k}_i(y,z) w_i (z) \mathrm{d} z\mathrm{d} y \right\vert=0
	\end{equation}
	uniformly on $\overline{\Omega}$ for all $1 \leq i \leq l$. Combining \eqref{equ:inf:1:e} and \eqref{equ:inf:2}, we have
	\begin{equation}\label{equ:inf:3}
		\lim\limits_{n_k \rightarrow +\infty}
		\left\vert (u_i)_{n_k}(x) - \int_{\Omega} \tilde{k}_i(x,y) \int_{\Omega} \tilde{k}_i(y,z) w_i (z) \mathrm{d} z\mathrm{d} y\right\vert=0
	\end{equation}
	uniformly on $\overline{\Omega}$ for all $ 1 \leq i \leq l$. Thus, by the uniqueness of the limit in weakly sense, 
	\begin{equation}\label{equ:inf:4}
		\int_{\Omega} \tilde{k}_i(x,y) \int_{\Omega} \tilde{k}_i(y,z) w_i (z) \mathrm{d} z\mathrm{d} y=w_i(x),~\forall x \in \overline{\Omega}, ~1 \leq i \leq l,
	\end{equation} 
	and $w_i$ is continuous on $\overline{\Omega}$.
	It then follows from Lemma \ref{lem:eventually} and \eqref{equ:pi:2} that $w_i(x)= v_i p_i(x)$ for some $v_i\geq 0$, $1 \leq i \leq l$. Furthermore, $\lim\limits_{n_k \rightarrow +\infty}\Vert \bm{u}_{n_k} - \bm{w} \Vert_X=0$ and $\Vert \bm{w} \Vert_X=1$ due to \eqref{equ:inf:3} and \eqref{equ:inf:4}, where $\bm{w} =(w_1,\cdots,w_l)^T$, which also imply that $\max_{1\leq i \leq n} v_i>0$.
	We integrate \eqref{equ:diff:eig:non-degen} over $\Omega$ to obtain
	$$
	\sum_{j=1}^l \int_{\Omega} m_{ij}(x) (u_j)_{n_k}(x) \mathrm{d} x
	= s(\mathcal{P}(\bm{d}_{n_k})) \int_{\Omega} (u_i)_{n_k}(x) \mathrm{d} x,~
	\forall 1 \leq i \leq l.
	$$  
	Letting $n_k \rightarrow +\infty$, we have
	$$
	\sum_{j=1}^l v_j \int_{\Omega} m_{ij}(x) p_j(x) \mathrm{d} x 
	= \bar{\lambda}  v_i,~
	\forall 1 \leq i \leq l,
	$$ 
	that is, 
	$
	\tilde{M} \bm{v} = \bar{\lambda} \bm{v}.
	$
	Therefore, the Perron-Frobenius theorem (see, e.g., \cite[Theorem 4.3.1]{smith2008monotone}) yields that $\bar{\lambda}=\tilde{\kappa}$. 
\end{proof}

\subsection{Partially degenerate case with large diffusion}
In this subsection, we study the asymptotic behavior of the spectral bound of $\mathcal{P}(\bm{d})$ as the diffusion coefficients go to positive infinity for the partially degenerate case. Recall that for any $\gamma > \eta_{22}$,
$$
B_{\gamma}(x)= M_{11} (x) + M_{12}(x)(\gamma I_2 - M_{22}(x))^{-1} M_{21}(x), \forall x\in \overline{\Omega}
$$
and $\tilde{B}_{\gamma}= (\tilde{b}_{ij,\gamma})_{l_1 \times l_1}$, where $\tilde{b}_{ij,\gamma} = \int_{\Omega} b_{ij,\gamma} (x) p_j(x) \mathrm{d} x$, $\forall 1 \leq i,j \leq l_1$.
\begin{lemma}\label{lem:bar_B}
	Assume that {\rm (H1)} and {\rm (H2)} hold. Then the following statements are valid:
	\begin{itemize}
		\item[\rm(i)] $\tilde{B}_{\gamma}$ is irreducible for all $\gamma > \eta_{22}$.
		\item[\rm (ii)] $s(\tilde{B}_{\gamma})$ is non-increasing and continuous with respect to $\gamma \in (\eta_{22},+\infty)$.
		\item[\rm (iii)] If $\lim\limits_{\gamma \rightarrow \eta_{22}^+ } s(\tilde{B}_{\gamma}) > \eta_{22}$, then there exists a unique $\gamma^*> \eta_{22}$ such that $s(\tilde{B}_{\gamma^*}) =\gamma^*$.
	\end{itemize}	 
\end{lemma}
\begin{proof}
	(i)  Since $p_j$ is a strongly positive continuous function on $\overline{\Omega}$ for each $1 \leq j \leq l$, there exists some $\epsilon_0>0$ and $x_0 \in \overline{\Omega}$ such that $\tilde{b}_{ij,\gamma} \geq \epsilon_0 b_{ij,\gamma}(x_0)$, $\forall 1 \leq i, j \leq l,~ \gamma > \eta_{22}$. Therefore, Lemma \ref{lem:A_split}(iv) implies that $B_{\gamma}(x_0)$, and hence, $\tilde{B}_{\gamma}$ is irreducible for any $\gamma > \eta_{22}$.
	
	(ii)
	We notice that if $\hat{\gamma}_1 > \hat{\gamma}_2 >\eta_{22}$, then 
	$B_{\hat{\gamma}_1} (x) \bm{v}_1 
	\leq B_{\hat{\gamma}_2} (x) \bm{v}_1,~\forall x \in \overline{\Omega},~ \bm{v}_1 \in \mathbb{R}^{l_1}_+$, so
	$$
	\tilde{B}_{\hat{\gamma}_1} \bm{v}_1 \leq \tilde{B}_{\hat{\gamma}_2} \bm{v}_1,~ \bm{v}_1 \in \mathbb{R}^{l_1}_+.
	$$
	In view of \cite[Theorem 1.1]{burlando1991monotonicity}, it is easy to see that $s(\tilde{B}_{\hat{\gamma}_1} ) \leq s(\tilde{B}_{\hat{\gamma}_2})$. The continuity can be derived by \cite[Proposition 2.7]{liang2019principal} and the matrix perturbation theory (see, e.g., \cite{steward1990matrices}).
	
	(iii)
	Define a function 
	$$
	f(\gamma): = s(\tilde{B}_{\gamma}) -\gamma, ~ \gamma > \eta_{22}.
	$$
	
	Obviously, $f(\gamma)$ is decreasing and continuous with respect to $\gamma \in (\eta_{22},+\infty)$.
	In view of $\lim\limits_{\gamma \rightarrow \eta_{22}^+ } s(\tilde{B}_{\gamma}) > \eta_{22}$, there exists $\gamma_1 > \eta_{22}$ close enough to $\eta_{22}$ such that $\gamma_2 :=s(\tilde{B}_{\gamma_1}) > \gamma_1> \eta_{22}$.
	We then have $s(\tilde{B}_{\gamma_2}) \leq s(\tilde{B}_{\gamma_1})$. Furthermore,
	$f(\gamma_1)= s(\tilde{B}_{\gamma_1}) - \gamma_1>0$ and 
	$f(\gamma_2)= s(\tilde{B}_{\gamma_2}) - \gamma_2= s(\tilde{B}_{\gamma_2}) - s(\tilde{B}_{\gamma_1})\leq0.$
	The intermediate value theorem yields that there is a unique number $\gamma^*>\gamma_1>\eta_{22}$ such that $f(\gamma^*)=0$, that is, $s(\tilde{B}_{\gamma^*})= \gamma^*$.
\end{proof}

\begin{lemma}\label{lem:s_e_p_d}
	Assume that {\rm (H1)--(H3)} and {\rm (H4$'$)} hold. Then the following statements are valid:
	\begin{itemize}
		\item[\rm (i)] $s_e(\mathcal{P}(\bm{d})) > \eta_{22}$.
		\item[\rm (ii)] If $\min_{1 \leq i \leq l_1} d_i \rightarrow +\infty$, then $s_e(\mathcal{P}(\bm{d})) \rightarrow \eta_{22}$.
	\end{itemize}	 
\end{lemma}

\begin{proof}
	Write $$
	M_{\bm{d}}(x):=
	\left(
	\begin{matrix}
		M_{11}(x) - D_1(x) & M_{12}(x)\\
		M_{21}(x)& M_{22}(x)
	\end{matrix}
	\right),~ \forall x \in \overline{\Omega},
	$$
	where $D_1(x)={\rm diag}(d_1\chi_1(x),\cdots,d_{l_1}\chi_{l_1}(x))$. For each $x \in \overline{\Omega}$, Lemma \ref{lem:A_split} implies that $s(M_{\bm{d}}(x))> s(M_{22}(x))$ and $\lim\limits_{\min_{1 \leq i \leq l_1} d_i \rightarrow + \infty} s(M_{\bm{d}}(x))=s(M_{22}(x))$.
	In view of Lemma \ref{lem:eta}, $s_e(\mathcal{P} (\bm{d}))=\max_{x\in \overline{\Omega}} s(M_{\bm{d}}(x))$. Parts (i) and (ii) can be derived by Lemma \ref{lem:eta22}.
\end{proof}

Let $\hat{\mathcal{D}}(\bm{d})$ and $\mathcal{B}_{\gamma}$ be two families of linear operators on $X_1$ defined by
$$
\hat{\mathcal{D}}(\bm{d}) \bm{u}^1:=	
([\mathcal{D}(\bm{d})]_1 u_1,[\mathcal{D}(\bm{d})]_2 u_2,\cdots,[\mathcal{D}(\bm{d})]_{l_1} u_{l_1})^T,~ \bm{u}^1 \in X_1,
$$
$$
[\mathcal{B}_{\gamma} \bm{u}^1](x):= B_{\gamma}(x) \bm{u}^1 (x),~ \bm{u}^1 \in X_1.
$$

We next introduce a family of linear operators $\hat{\mathcal{P}}(\bm{d},\gamma)$ on $X_1$ defined by
$$\hat{\mathcal{P}}(\bm{d},\gamma):=\hat{\mathcal{D}}(\bm{d}) +\mathcal{B}_{\gamma}$$
and present their properties.
\begin{lemma}\label{lem:cp_gamma}
	Assume that {\rm (H1)--(H3)} and {\rm (H4$'$)} hold. Then the following statements are valid:
	\begin{itemize}
		\item[\rm (i)] For any given $\gamma > \eta_{22}$, there exists $\hat{d}>0$ such that $s(\hat{\mathcal{P}}(\bm{d},\gamma))$ is the principal eigenvalue of $\hat{\mathcal{P}}(\bm{d},\gamma)$ if $\min_{1\leq i \leq l_1} d_i \geq \hat{d} $. Moreover, $ s(\hat{\mathcal{P}}(\bm{d},\gamma)) \rightarrow s(\tilde{B}_{\gamma})$ as $\min_{1 \leq i \leq l_1} d_i \rightarrow +\infty$.
		\item[\rm (ii)]$s(\hat{\mathcal{P}} (\bm{d}, \gamma))$ and $s_e(\hat{\mathcal{P}} (\bm{d}, \gamma))$ are continuous and non-increasing with respect to $\gamma \in (\eta_{22},+\infty)$.
		\item[\rm (iii)] If $s(\hat{\mathcal{P}} (\bm{d},\gamma_0))>\gamma_0$ for some $\gamma_0 > \eta_{22}$, then there exists a unique $\mu>\gamma_0$ such that $s(\hat{\mathcal{P}}(\bm{d},\mu))=\mu$.
	\end{itemize}	
\end{lemma}

\begin{proof}
	Part (i) follows from Lemma \ref{lem:D:large:non-degen} and Theorem \ref{thm:D:inf:non-degen}.
	Part (ii) has been shown in Lemma \ref{lem:L_F}.
	Part (iii) can be derived by the arguments similar to those in Lemma \ref{lem:bar_B}(iii).
\end{proof}

\begin{lemma}\label{lem:D:large:degen}
	Assume that {\rm (H1)--(H3)} and {\rm (H4$'$)} hold and $\lim\limits_{\gamma \rightarrow \eta_{22}^+} s(\tilde{B}_{\gamma}) > \eta_{22}$. Choose $\gamma^*> \eta_{22}$ mentioned in Lemma \ref{lem:bar_B} such that $s(\tilde{B}_{\gamma^*}) =\gamma^*$.
	For any $\epsilon_0>0$ with $\gamma^* - 2 \epsilon_0 > \eta_{22}$, there exists a positive constant $\hat{d}$ such that if $\min_{1\leq i \leq l_1} d_i \geq \hat{d} $, then
	$$\eta_{22} +\epsilon_0 \leq s(\mathcal{P}(\bm{d})). $$
	Moreover, 
	$s(\mathcal{P}(\bm{d}))$ is the principal eigenvalue of $\mathcal{P}(\bm{d})$ if $\min_{1\leq i \leq l_1} d_i \geq \hat{d} $. 
\end{lemma}
\begin{proof}
	For any given $\epsilon_0 > 0$ with $\gamma^* - 2 \epsilon_0 > \eta_{22}$, in view of Lemma \ref{lem:D:large:non-degen}, we can choose $\hat{d}_0$ large enough such that if $\min_{1\leq i \leq l_1} d_i \geq \hat{d}_0$, then
	$$s_e(\hat{\mathcal{P}}(\bm{d},\eta_{22}+\epsilon_0)) \leq \eta_{22}.$$
	According to Lemma \ref{lem:cp_gamma}(i), there exists some $\hat{d}_1 \geq 0$ such that if $\min_{1\leq i \leq l_1} d_i \geq \hat{d}_1$, then
	$$\vert s(\hat{\mathcal{P}}(\bm{d},\gamma^*)) - s(\tilde{B}_{\gamma^*})\vert 
	\leq \epsilon_0.$$ 
	Now we assume that $\min_{1\leq i \leq l_1} d_i \geq \hat{d}:= \max (\hat{d}_0,\hat{d}_1)$. It then follows from $s(\tilde{B}_{\gamma^*})=\gamma^* > \eta_{22} + 2\epsilon_0 $ and Lemma \ref{lem:cp_gamma}(ii) that
	$$
	s(\hat{\mathcal{P}}(\bm{d},\eta_{22}+\epsilon_0)) 
	\geq s(\hat{\mathcal{P}}(\bm{d},\gamma^*)) 
	\geq s(\tilde{B}_{\gamma^*}) - \epsilon_0
	> \eta_{22} +\epsilon_0.
	$$
	Lemma \ref{lem:cp_gamma}(iii) yields that $s(\hat{\mathcal{P}}(\bm{d},\mu(\bm{d})))=\mu(\bm{d})$ for some $\mu(\bm{d}) > \eta_{22} + \epsilon_0$. Thus,
	$$
	s(\hat{\mathcal{P}}(\bm{d},\mu(\bm{d})))
	=\mu(\bm{d})>\eta_{22} +\epsilon_0
	\geq s_e(\hat{\mathcal{P}}(\bm{d},\eta_{22}+\epsilon_0)) 
	\geq s_e(\hat{\mathcal{P}}(\bm{d},\mu(\bm{d}))
	$$
	by Lemma \ref{lem:cp_gamma}(ii).
	Proposition \ref{prop:equivalent:condition}(iv) implies that $\mu(\bm{d})=s(\mathcal{P}(\bm{d}))$ is the principal eigenvalue of $\mathcal{P}(\bm{d})$. This completes the proof.
\end{proof}

Summarizing Lemmas \ref{lem:D:large:non-degen} and \ref{lem:D:large:degen}, we obtain another sufficient condition for the existence of the principal eigenvalue.
\begin{theorem}\label{thm:D:large:exist}
	Assume that {\rm (H1)--(H3)} hold. If, in addition, {\rm (H4)} holds or {\rm (H4$'$)} and $\lim\limits_{\gamma \rightarrow \eta_{22}^+} s(\tilde{B}_{\gamma}) > \eta_{22}$ hold true. Then there exists $\hat{d}$ large enough such that
	$s(\mathcal{P}(\bm{d}))$ is the principal eigenvalue of $\mathcal{P}(\bm{d})$ if $\min_{1\leq i \leq l_1} d_i \geq \hat{d}$. 
\end{theorem}

By repeating the arguments in the proof of Lemma \ref{lem:s:bounded}, we have the following results.
\begin{lemma}\label{lem:den:upper}
	Assume that {\rm (H1)--(H3)} and {\rm (H4$'$)} hold.
	We obtain the estimate $ s(P(\bm{d})) \leq C $, where the constant $C$ is independent of $\bm{d}$.
\end{lemma} 
Finally, we are in a position to prove the main result of this subsection.
\begin{theorem}\label{thm:inf:degen}
	Assume that {\rm (H1)--(H3)} and {\rm (H4$'$)} hold. Then the following statements are valid:
	\begin{itemize}
		\item[\rm (i)] If $\lim\limits_{\gamma \rightarrow \eta_{22}^+ } s(\tilde{B}_{\gamma}) >\eta_{22}$, then there exists a unique $\gamma^*> \eta_{22}$ such that $s(\tilde{B}_{\gamma^*}) =\gamma^*$ and 
		$$s(\mathcal{P}(\bm{d})) \rightarrow \gamma^* \text{ as } \min_{1 \leq i \leq l_1} d_i \rightarrow +\infty.$$
		\item[\rm (ii)] If $\lim\limits_{\gamma \rightarrow \eta_{22}^+ } s(\tilde{B}_{\gamma}) \leq \eta_{22}$, then
		$$s(\mathcal{P}(\bm{d})) \rightarrow \eta_{22} \text{ as } \min_{1 \leq i \leq l_1} d_i \rightarrow +\infty.$$
	\end{itemize}
\end{theorem}

\begin{proof}
	(i)
	Choose $\epsilon_0 >0$ such that $\gamma^* - 2 \epsilon_0 > \eta_{22}$.
	By Lemma \ref{lem:D:large:degen}, there exists $\hat{d}>0$ such that if $\min_{1\leq i \leq l_1} d_i \geq \hat{d} $, then $\mathcal{P}(\bm{d})$ has the principal eigenvalue.
	It suffices to show that for any sequence $\bm{d}_n=((d_1)_n,\cdots,(d_l)_n)$, there is a subsequence $\bm{d}_{n_k}$ such that $s(\mathcal{P}(\bm{d}_{n_k})) \rightarrow \gamma^*$ as $\min_{1\leq i \leq l_1} (d_{i})_{n_k} \rightarrow +\infty$. Without loss of generality, we assume that $\min_{1\leq i \leq l_1} (d_i)_n \geq \hat{d} $ for all $n \geq 1$. Let $\bm{u}_n=((u_1)_n,\cdots,(u_l)_n)^T=((\bm{u}^1)_{n}^T,(\bm{u}^2)_{n}^T)^T$ be the positive eigenvector of $\mathcal{P}(\bm{d}_{n})$ corresponding to $s(\mathcal{P}(\bm{d}_{n}))$ with normalizing $\bm{u}_n$ by $\Vert \bm{u}_n \Vert_X=1$, that is, 
	$\max_{1\leq i \leq l,x\in \overline{\Omega}} (u_i)_n(x) =1$. Thus,
	\begin{equation}\label{equ:diff:eig:degen}
		\begin{cases}
			s(\mathcal{P}(\bm{d}_n)) (u_i)_n(x)=(d_i)_n\left[ \int_{\Omega} k_i(x,y) (u_i)_n(y) \mathrm{d} y - \chi_i(x)(u_i)_n(x) \right]\\ \hspace*{35mm}+\sum_{j=1}^l m_{ij}(x) (u_j)_n(x),~
			&1 \leq i \leq l_1, x\in \overline{\Omega},\\
			s(\mathcal{P}(\bm{d}_n)) (u_i)_n(x)=\sum_{j=1}^l m_{ij}(x) (u_j)_n(x),~
			&l_1+1 \leq i \leq l, x\in \overline{\Omega}.
		\end{cases}
	\end{equation}
	For any $ 1 \leq i \leq l_1$, we divide the $i$-th row of \eqref{equ:diff:eig:degen} by $(d_i)_n$ to obtain
	$$
	\left[ \int_{\Omega} k_i(x,y) (u_i)_n(y) \mathrm{d} y - \chi_i(x)(u_i)_n(x) \right] +\frac{1}{(d_i)_n}
	\left(
	\sum_{j=1}^l  m_{ij}(x) (u_j)_n(x)- s(\mathcal{P}(\bm{d}_n)) (u_i)_n(x) 
	\right)=0.
	$$
	Notice that Lemmas \ref{lem:D:large:degen} and \ref{lem:den:upper} yield that
	$$
	\frac{1}{(d_i)_n}
	\left(
	\sum_{j=1}^l  m_{ij}(x) (u_j)_n(x)- s(\mathcal{P}(\bm{d}_n)) (u_i)_n(x) 
	\right)
	\rightarrow 0 \text{ as }(d_i)_n \rightarrow +\infty,~\forall 1 \leq i \leq l_1.
	$$
	Moreover, there exists a subsequence $n_k$ such that $s(\mathcal{P}(\bm{d}_{n_k}))$ converges to some $\bar{\lambda} \geq \eta_{22} + \epsilon_0$ and $(u_i)_{n_k}$ weakly converges to some $w_i$ in $L^2(\Omega)$ as $n_k \rightarrow +\infty$ for all $1\leq i\leq l$. Analysis similar to those in Theorem \ref{thm:D:inf:non-degen} shows that for all $1 \leq i \leq l_1$
	$$
	(u_i)_{n_k}(x) \rightarrow w_i (x)=v_i p_i(x) \text{ uniformly on } \overline{\Omega}
	\text{ as } n_k \rightarrow +\infty,
	$$
	for some constant $v_i$. Write $\bm{w}^1=(w_1,\cdots,w_{l_1} )^T \in X_1$ and $\bm{v}^1 =(v_1,\cdots,v_{l_1})^T \in \mathbb{R}^l$. Thanks to $\bar{\lambda} \geq \eta_{22} + \epsilon_0$ and
	$$	
	A_{21}(x) (\bm{u}^1)_{n_k}(x) + A_{22}(x) (\bm{u}^2)_{n_k}(x) =s(\mathcal{P}(\bm{d}_{n_k})) (\bm{u}^2)_{n_k} (x),
	$$
	it then follows that
	$(\bm{u}^2)_{n_k}$ uniformly converges to some 
	$$\bm{w}^2(x)=(w_{l_1+1}(x),\cdots, w_{l}(x)) =(\overline{\lambda} - A_{22} (x))^{-1}A_{21}(x) \bm{w}^1(x)$$
	as $n_k \rightarrow +\infty$.
	Therefore, $\lim\limits_{n_k \rightarrow +\infty}\Vert \bm{u}_{n_k} - \bm{w} \Vert_X=0$ and $\Vert \bm{w} \Vert_X=1 $ where $\bm{w}=((\bm{w}^1)^T,(\bm{w}^2)^T)^T$.
	We integrate the $i$-th equation of \eqref{equ:diff:eig:degen} over $\Omega$ for $1\leq i \leq l_1$ to obtain
	$$
	\begin{cases}
		\sum_{j=1}^l \int_{\Omega} m_{ij}(x) (u_j)_{n_k}(x) \mathrm{d} x
		= s(\mathcal{P}(\bm{d}_{n_k})) \int_{\Omega} (u_i)_{n_k}(x) \mathrm{d} x,~
		\forall 1 \leq i \leq l_1,\\
		\sum_{j=1}^l m_{ij}(x) (u_j)_{n_k}(x)= s(\mathcal{P}(\bm{d}_{n_k})) (u_i)_{n_k}(x),~
		l_1+1 \leq i \leq l, x\in \overline{\Omega}.
	\end{cases}
	$$  
	Letting $n_k \rightarrow +\infty$, we have
	$$
	\begin{cases}
		\sum_{j=1}^{l_1} v_j \int_{\Omega} m_{ij}(x) p_j(x) \mathrm{d} x 
		+\sum_{j=l_1+1}^{l} \int_{\Omega} m_{ij}(x) w_j(x) \mathrm{d} x
		= \bar{\lambda} v_i,~ 1 \leq i \leq l_1,\\
		\sum_{j=1}^{l_1} m_{ij}(x) v_j p_j(x) 
		+\sum_{j=l_1+1}^{l} m_{ij}(x) w_j(x)= \bar{\lambda} w_i(x),~
		l_1+1 \leq i \leq l, x\in \overline{\Omega}.
	\end{cases}
	$$
	This implies that
	$
	\tilde{B}_{\bar{\lambda}} \bm{v}^1 =\bar{\lambda} \bm{v}^1
	$
	due to $\bar{\lambda} \geq \eta_{22} + \epsilon_0$.
	Then the Perron-Frobenius Theorem (see, e.g., \cite[Theorem 4.3.1]{smith2008monotone}) leads to that $\bar{\lambda} = \gamma^*$.
	
	(ii) 
	We now consider the case $\lim\limits_{\gamma \rightarrow \eta_{22}^+ } s(\tilde{B}_{\gamma}) \leq \eta_{22}$. Lemma \ref{lem:s_e_p_d}(i) leads to
	$$
	s(\mathcal{P}(\bm{d})) \geq s_e(\mathcal{P}(\bm{d})) > \eta_{22},~ \forall \bm{d} \in \mathbb{R}^l_+.
	$$ 
	Choose $\alpha_0 >0$ such that 
	$\lim\limits_{\gamma \rightarrow \eta_{22}^+ } s(\tilde{B}_{\gamma}+ \alpha_0 I_1 ) = \eta_{22}$, where $I_1$ is an $l_1$-dimensional identity matrix. For any $\alpha > \alpha_0$, write
	$$
	M_{\alpha}(x):=
	\left(
	\begin{matrix}
		M_{11}(x) + \alpha I_1& M_{12}(x)\\
		M_{21}(x)& M_{22}(x)
	\end{matrix}
	\right),~ \forall x \in \overline{\Omega},
	$$
	let $\mathcal{M}_{\alpha}:~X \rightarrow X$ be a family of linear operators defined by:
	$$
	[\mathcal{M}_{\alpha} \bm{u} ](x):=M_{\alpha}(x) \bm{u}(x),~ x \in \overline{\Omega}, ~ \bm{u} \in X,
	$$
	and define
	$$
	\mathcal{P}_{\alpha}(\bm{d}):= \mathcal{D}(\bm{d}) + \mathcal{M}_{\alpha},~\alpha > \alpha_0.
	$$
	For any $\alpha > \alpha_0$, we have $\lim\limits_{\gamma \rightarrow \eta_{22}^+ } s(\tilde{B}_{\gamma}+ \alpha I_1 ) > \eta_{22}$, and hence, there exists $\gamma^* (\alpha) > \eta_{22}$ such that 
	$s(\tilde{B}_{\gamma^* (\alpha)} +\alpha) =\gamma^* (\alpha)$ using Lemma \ref{lem:bar_B}(iii) with $\tilde{B}_{\gamma}$ substituted by $\tilde{B}_{\gamma}+ \alpha I_1$.	
	We next claim that
	\begin{equation}
		\gamma^*(\alpha) \leq \eta_{22} +\alpha - \alpha_0,~ \forall \alpha > \alpha_0.
	\end{equation}

	Next, we prove the claim and fix $\alpha > \alpha_0$. Write $f(\gamma):=s(\tilde{B}_{\gamma}+ \alpha I_1 ) - \gamma, ~ \gamma > \eta_{22}$. By Lemma \ref{lem:bar_B}(ii) and
	$\lim\limits_{\gamma \rightarrow \eta_{22}^+ } s(\tilde{B}_{\gamma}+ \alpha_0 I_1 ) = \eta_{22}$, it is easy to see that
	$
	s(\tilde{B}_{\eta_{22}+ \alpha - \alpha_0}+ \alpha_0 I_1) \leq \eta_{22}.
	$
	This implies that
	$$
	s(\tilde{B}_{\eta_{22}+ \alpha - \alpha_0}+ \alpha I_1 ) \leq \eta_{22} + \alpha - \alpha_0.
	$$
	In addition,
	$$
	\lim\limits_{\gamma \rightarrow \eta_{22}^+ } s(\tilde{B}_{\gamma}+ \alpha I_1 ) 
	= \eta_{22} + \alpha - \alpha_0 
	> \eta_{22}.
	$$
	Consequently, $f(\eta_{22} +\alpha -\alpha_0) \leq 0$, $f(\gamma^*(\alpha))=0$ and $\lim\limits_{\gamma \rightarrow \eta_{22}^+} f(\gamma) >0$. Noting that $f(\gamma)$ is continuous and decreasing with respect to $\gamma \in (\eta_{22},+\infty)$ by Lemma \ref{lem:bar_B}, we have
	$$
	\gamma^*(\alpha) \leq \eta_{22} + \alpha - \alpha_0,
	$$
	which derive our claim.
	
	For any given $\epsilon >0$, choose $\alpha_1 > \alpha_0>0$ with
	$\vert \alpha_1 - \alpha_0 \vert \leq \frac{\epsilon}{2}$. The above claim yields that $\vert \gamma^*(\alpha_1) - \eta_{22} \vert \leq \frac{\epsilon}{2}$.
	Obviously,
	$$\lim\limits_{\gamma \rightarrow \eta_{22}^+ } s(\tilde{B}_{\gamma}+ \alpha_1 I_1 ) = \eta_{22} + \alpha_1 - \alpha_0 > \eta_{22}.$$ In view of $\int_{\Omega} p_i(x) \mathrm{d} x=1$, $\forall i=1,\cdots, l_1$, using (i) with $\tilde{B}_{\gamma}$ replaced by $\tilde{B}_{\gamma}+ \alpha_1 I_1$,  there exists $\hat{d}>0$ such that if $\min_{1\leq i \leq l_1} d_i \geq \hat{d}$, then 
	$$
	\vert s(\mathcal{P}_{\alpha_1}(\bm{d})) - \gamma^*(\alpha_1) \vert \leq \frac{\epsilon}{2}.
	$$ 
	Moreover, \cite[Theorem 1.1]{burlando1991monotonicity} implies that $s(\mathcal{P}(\bm{d})) - s(\mathcal{P}_{\alpha_1}(\bm{d})) \leq 0$.
	We conclude that 
	$$
	\begin{aligned}
		0 \leq s(\mathcal{P}(\bm{d})) - \eta_{22} 
		&\leq s(\mathcal{P}(\bm{d})) - s(\mathcal{P}_{\alpha_1}(\bm{d})) 
		+ s(\mathcal{P}_{\alpha_1}(\bm{d})) - \gamma^*(\alpha_1) 
		+ \gamma^*(\alpha_1) - \eta_{22} \\
		&\leq 0 + \frac{\epsilon}{2} + \frac{\epsilon}{2} 
		= \epsilon
	\end{aligned}
	$$ 
	if $\min_{1\leq i \leq l_1} d_i \geq \hat{d}$. This completes the proof.
\end{proof}

\section{An application}\label{sec:app}

In this section, we employ the above result to analyze the asymptotic behavior of the basic reproduction ratio for a infection model with cell-to-cell transmission and nonlocal viral dispersal. The authors in \cite{Shu2020JMPA,Wang2014AA,Wang2016JMAA} investigated a viral infection model using a partially degenerate reaction-diffusion system. To account for the rapid movement of viruses, in this paper, we assume that the diffusion process exhibits nonlocal dispersion. Thus the modified model can be expressed as follows:
\begin{equation}\label{equ:model:VSI}
	\begin{cases}
		\frac{\partial V(x,t)}{\partial t}= d [\int_{\Omega} k(x,y)V(y,t)\mathrm{d} y -\int_{\Omega} k(y,x) V(x,t) \mathrm{d} y] +r(x) I(x,t)-m(x) V(x,t),\\
		\frac{\partial S(x,t)}{\partial t}=n(x,S(x,t))-f(x,V(x,t),S(x,t))-g(x,S(x,t),I(x,t)),\\
		\frac{\partial I(x,t)}{\partial t}= f(x,V(x,t),S(x,t))+g(x,S(x,t),I(x,t))-b(x) I(x,t)
	\end{cases}
\end{equation}
for all $x\in \overline{\Omega}$, $t>0$.
Here, $V(x,t)$, $S(x,t)$, $I(x,t)$  denote the populations of free virus particles, susceptible target cells and infected target cells at location $x$ and time $t$, respectively; $d$ is the diffusion coefficients and $k(x,y)$ is a non-negative continuous function of $(x,y) \in \overline{\Omega} \times \overline{\Omega}$ with $k(x,x)>0$ for all $x \in \overline{\Omega}$; $r(x)>0$ is the rate of virus production by the lysis of infected cells; $m(x)>0$ and $b(x)>0$ are the death rate of free virues and infected cells; $n(x,S)$, $f(x,V,S)$, $g(x,S,I)$ are the cell reproduction function,  cell-free transmission function, and cell-to-cell transmission function, respectively.  We further assume the following:

\begin{itemize}
	\item[(A1)] $n \in C^1(\overline{\Omega}\times \mathbb{R}_+, \mathbb{R})$ and $\partial_S n(x,S) \leq 0$ for all $x \in \overline{\Omega}$ and $S \geq 0$. Moreover, there exists a unique $S^* \in C(\overline{\Omega},\mathbb{R})$ such that $S^*(x)>0$ and $n(x, S^*(x))=0$ for all $ x \in \overline{\Omega}$.
	\item[(A2)] $f,g \in C^1(\overline{\Omega} \times \mathbb{R}_+ \times \mathbb{R}_+,\mathbb{R}) $ and the partial derivatives $\frac{\partial g}{\partial I}(x,S^*(x),0)$ and $\frac{\partial f}{\partial V} (x,S^*(x),0)$ are positive for all $x \in \overline{\Omega}$. Moreover, $f(x,V,S)=0$ if and only if $VS=0$, $g(x,S,I)=0$ if and only if $SI=0$.
\end{itemize}

System \eqref{equ:model:VSI} has a unique infected-free steady state $(0,S^*(x),0)$. For convenience, write 
$$\beta_d (x):= \frac{\partial g}{\partial I}\left(x,S^*(x),0\right),\quad \beta_i(x):=\frac{\partial f}{\partial V} \left(x,S^*(x),0\right).$$

We linearize system \eqref{equ:model:VSI} at the  infected-free steady state $(0,S^*(x),0)$ to obtain the following cooperative linear system:

\begin{equation}\label{equ:model:VI}
	\begin{cases}
		\frac{\partial V(x,t)}{\partial t}=d [\int_{\Omega} k(x,y)V(y,t)\mathrm{d} y -\int_{\Omega} k(y,x) V(x,t) \mathrm{d} y] +r(x) I(x,t)-m(x) V(x,t), \\
		\frac{\partial I(x,t)}{\partial t}= \beta_i(x)V(x,t)+\beta_d(x) I(x,t)-b(x) I(x,t)
	\end{cases}
\end{equation}
for all $x \in \overline{\Omega}$, $t>0$.
For any $d \geq 0$, define two bounded linear operators  $\mathcal{L}(d):C(\overline{\Omega},\mathbb{R}) \rightarrow C(\overline{\Omega},\mathbb{R})$ and  $B(d): C(\overline{\Omega},\mathbb{R}^2) \rightarrow C(\overline{\Omega},\mathbb{R}^2)$ as follows: 
$$\mathcal{L}(d) \phi=d \left[\int_{\Omega} k(\cdot,y)\phi(y)\mathrm{d} y -\int_{\Omega} k(y,\cdot) \phi(\cdot) \mathrm{d} y\right],~
B(d)=
\left(
\begin{array}{cc}
	\mathcal{L}(d) - m(\cdot)& r(\cdot)  \\
	0& -b(\cdot)
\end{array}
\right).$$

Let  $F: C(\overline{\Omega},\mathbb{R}^2) \rightarrow C(\overline{\Omega},\mathbb{R}^2)$ be a bounded linear operator defined by
$$
F
=
\left(
\begin{array}{cc}
	0&0  \\
	\beta_i(\cdot)& \beta_d (\cdot) 
\end{array}
\right).
$$
According to \cite[Section 3]{thieme2009spectral}, for any $d  \geq 0$, the next generation operator is $-F[B(d)]^{-1}$ and the basic reproduction ratio is $\mathcal{R}_0(d)=r(-F[B(d)]^{-1})$. Write $\hat{\mathcal{R}}_0:=\max_{x \in \overline{\Omega}}\frac{\beta_d(x)}{b(x)}$. By Lemma \ref{lem:p}, there exists a strongly positive continuous function $p$ on $\overline{\Omega}$ such that $\int_{\Omega} k(x,y) p(y) \mathrm{d} y = \int_{\Omega} k(y,x)  \mathrm{d} y p(x)$, $\forall x \in \overline{\Omega}$ and $\int_{\Omega} p(x) \mathrm{d} x=1$. 
For any $\mu >\hat{\mathcal{R}}_0$, write
$$
Q(\mu):=\int_{\Omega} \left[-m(x)+ \frac{r(x)\beta_i(x)}{\mu b(x) - \beta_d(x)}\right]p(x) \mathrm{d} x.
$$
Then we can present the main result of this section:
\begin{theorem}
	Then the following statements are valid:
	\begin{itemize}
		\item[\rm (i)] $\lim\limits_{d \rightarrow 0^+} \mathcal{R}_0(d)=\max_{x \in \overline{\Omega}} \left(\frac{\beta_d(x)}{b(x)} + \frac{\beta_i(x)r(x)}{b(x) m(x)}\right)$.
		\item[\rm (ii)] If $Q(\hat{\mathcal{R}}_0):=\lim\limits_{\mu \rightarrow \hat{\mathcal{R}}_0^+}Q(\mu)>0$, then there exists an unique $\tilde{\mathcal{R}}_0>\hat{\mathcal{R}}_0$ such that $Q(\tilde{\mathcal{R}}_0)=0$ and $\lim\limits_{d \rightarrow +\infty} \mathcal{R}_0(d)=\tilde{\mathcal{R}}_0$. If $\lim\limits_{\mu \rightarrow \hat{\mathcal{R}}_0^+}Q(\mu) \leq 0$, then $\lim\limits_{d \rightarrow +\infty} \mathcal{R}_0(d)=\hat{\mathcal{R}}_0$.
	\end{itemize}
\end{theorem}

\begin{proof}
	(i)
	For any $\mu>0$, $d\geq 0$, define $H(\mu,d)=s(B(d)+ \frac{1}{\mu} F)$. Thanks to Theorem \ref{thm:conti:KA}, $H(\mu,d)$ is continuous with respect to  $(\mu,d) \in (\mathbb{R}_+\setminus\{0\}) \times \mathbb{R}_+$.  It follows from \cite[Lemma 2.5]{Zhang2021SIMA} that for any $d\geq 0$, $\mathcal{R}_0 (d)>0$ and $H(\mathcal{R}_0(d),d)=0$, $H(\mu,d)<0$ for all $\mu >\mathcal{R}_0(d)$ and $H(\mu,d)>0$ for all $\mu \in (0,\mathcal{R}_0(d))$. According to \cite[Lemma 2.5]{Zhang2022SCM}, we obtain  $\lim\limits_{d \rightarrow 0^+} \mathcal{R}_0(d)= \mathcal{R}_0(0)$ and $H(\mathcal{R}_0(0),0)=0$. By solving the equation $s(B(0)+ \frac{1}{\mathcal{R}_0(0)} F)=0$, we conclude that $\mathcal{R}_0(0)=\max_{x \in \overline{\Omega}} \left(\frac{\beta_d(x)}{b(x)} + \frac{\beta_i(x)r(x)}{b(x) m(x)}\right)$.
	
	(ii) For any $\mu>0$, $\theta>0$, define $\hat{H}(\mu,\theta)=s(B(\frac{1}{\theta})+\frac{1}{\mu}F )$. By \cite[Lemma 2.5]{Zhang2021SIMA} again, we obtain that for any $\theta > 0$, $\mathcal{R}_0 (\frac{1}{\theta})>0$ and $\hat{H}(\mathcal{R}_0(\frac{1}{\theta}),\theta)=0$, $\hat{H}(\mu,\theta)<0$ for all $\mu >\mathcal{R}_0(\frac{1}{\theta})$ and $\hat{H}(\mu,\theta)>0$ for all $\mu \in (0,\mathcal{R}_0(\frac{1}{\theta}))$. For any $\mu>0$, write $\eta_{22}(\mu)=\max_{x \in \overline{\Omega}} [\frac{1}{\mu} \beta_d(x)-b(x)]=s(\frac{1}{\mu} \beta_d-b)$. According to Theorem \ref{thm:inf:degen}, for any $\mu >0$, there exists $\tilde{\lambda}(\mu) \geq \eta_{22}(\mu)$ such that $\hat{H}(\mu,\theta) \rightarrow \tilde{\lambda}(\mu)$ as $\theta \rightarrow 0^+$. We remark that $\tilde{\lambda}(\mu)$ is non-increasing with respect to $\mu >0$.
	
	   For any $\mu>0$ and $\gamma>\eta_{22} (\mu)$, define
	$$
	\tilde{B}(\mu,\gamma)=\int_{\Omega} \left[-m(x)+ \frac{r(x)\beta_i(x)}{\mu\gamma + \mu b(x) - \beta_d(x)}\right]p(x) \mathrm{d} x.
	$$
	We remark that $\eta_{22}(\hat{\mathcal{R}}_0)=0$.
	Notice that $\tilde{B}(\mu,0)=Q(\mu)$ is decreasing with respect to $\mu \in (\hat{\mathcal{R}}_0,+\infty)$ and $Q(\mu)<0$ when $\mu$ is large enough. 
	
	In the case where $Q(\hat{\mathcal{R}}_0)>0$,  there exists a unique $\tilde{\mathcal{R}}_0> \hat{\mathcal{R}}_0$ such that $\tilde{B}(\tilde{\mathcal{R}}_0,0)=Q(\tilde{\mathcal{R}}_0)=0$. Since $\tilde{B}(\tilde{\mathcal{R}}_0,\gamma)$ is decreasing with respect to $\gamma>\eta_{22}(\tilde{\mathcal{R}}_0)$ and $\eta_{22}(\tilde{\mathcal{R}}_0)<\eta_{22}(\hat{\mathcal{R}}_0)=0$, we have 
	$$\lim\limits_{\gamma \rightarrow \eta_{22}(\tilde{\mathcal{R}}_0)^+}\tilde{B}(\tilde{\mathcal{R}}_0,\gamma)>\tilde{B}(\tilde{\mathcal{R}}_0,0)=0>\eta_{22}(\tilde{\mathcal{R}}_0).$$
	By Theorem \ref{thm:inf:degen}(i) with $\tilde{B}_{\gamma}=\tilde{B}(\tilde{\mathcal{R}}_0,\gamma)$,  we obtain that $\tilde{\lambda}(\tilde{\mathcal{R}}_0)=\tilde{B}(\tilde{\mathcal{R}}_0,0)=0$.
	
	We next prove that $\tilde{\lambda}(\mu)<0$ for all $\mu >\tilde{\mathcal{R}}_0$ and $\tilde{\lambda}(\mu)>0$ for all $\mu \in (0,\tilde{\mathcal{R}}_0)$. Since $\tilde{B}(\mu, \gamma)$ is continuous with respect to $\mu >0$ and $\gamma> \eta_{22}(\mu)$ and decreasing with respect to $\mu >0$, and $\eta_{22}(\mu)$ is continuous on $(0,+\infty)$, we then have $\eta_{22}(\mu)<\tilde{B}(\mu,0)<0$ for $\mu \in (\tilde{\mathcal{R}}_0,\mu_1)$ for some $\mu_1>\tilde{\mathcal{R}}_0$. By Lemma \ref{lem:bar_B}, for any $\mu \in (\tilde{\mathcal{R}}_0,\mu_1)$, there is $\gamma^*(\mu)<0$ such that $\tilde{B}(\mu,\gamma^*(\mu))=\gamma^*(\mu)>\eta_{22}(\mu)$, and $\tilde{\lambda}(\mu)=\gamma^*(\mu)<0$ for all $\mu \in (\tilde{\mathcal{R}}_0,\mu_1)$. Notice that $\tilde{\lambda}(\mu)$ is non-increasing with respect to $\mu\in (0,+\infty)$. It then follows that $\tilde{\lambda}(\mu)<0$ for all $\mu >\tilde{\mathcal{R}}_0$. Similarly,  $\tilde{\lambda}(\mu)>0$ for all $\mu <\tilde{\mathcal{R}}_0$. According to \cite[Lemma 2.5]{Zhang2022SCM}, we obtain  $\lim\limits_{\theta \rightarrow 0^+} \mathcal{R}_0(\frac{1}{\theta})= \tilde{\mathcal{R}}_0$.
	
	In the case where $Q(\hat{\mathcal{R}}_0)\leq 0$, in view of $\eta_{22}(\hat{\mathcal{R}}_0)=0$, it follows that
	$$\lim\limits_{\gamma \rightarrow \eta_{22}(\hat{\mathcal{R}}_0)^+}\tilde{B}(\hat{\mathcal{R}}_0,\gamma)=\tilde{B}(\hat{\mathcal{R}}_0,0)=Q(\hat{\mathcal{R}}_0) \leq 0=\eta_{22}(\hat{\mathcal{R}}_0).$$
	Thanks to Theorem \ref{thm:inf:degen}(ii) with $\tilde{B}_{\gamma}=\tilde{B}(\hat{\mathcal{R}}_0,\gamma)$, we have $\tilde{\lambda}(\hat{\mathcal{R}}_0)=0$.
	
	It suffices to prove that $\tilde{\lambda}(\mu)<0$ for all $\mu >\hat{\mathcal{R}}_0$ and $\tilde{\lambda}(\mu)>0$ for all $\mu \in (0,\hat{\mathcal{R}}_0)$. Clearly, $\tilde{\lambda}(\mu)\geq \eta_{22}(\mu)>0$ for all $\mu \in (0,\hat{\mathcal{R}}_0)$. Notice that $\tilde{\lambda}(\mu)$ is non-increasing with respect to $\mu \in (0,+\infty)$. Thus $\tilde{\lambda}(\mu) \leq 0$ for all $\mu >\hat{\mathcal{R}}_0$. Suppose that there exists $\bar{\mathcal{R}}_0>\hat{\mathcal{R}}_0$ such that $\tilde{\lambda}(\bar{\mathcal{R}}_0)=0$. It is easy to see that $\eta_{22}(\bar{\mathcal{R}}_0) <\eta_{22}(\hat{\mathcal{R}}_0) =0$. In the case where $\lim\limits_{\gamma \rightarrow \eta_{22}(\bar{\mathcal{R}}_0)^+}\tilde{B}(\bar{\mathcal{R}}_0,\gamma) \leq \eta_{22}(\bar{\mathcal{R}}_0)$, by Theorem \ref{thm:inf:degen}(ii) with $\tilde{B}_{\gamma}=\tilde{B}(\bar{\mathcal{R}}_0,\gamma)$, we have $\tilde{\lambda}(\bar{\mathcal{R}}_0)=\eta_{22}(\bar{\mathcal{R}}_0)<0$, which derives a contradiction. In the case where $\lim\limits_{\gamma \rightarrow \eta_{22}(\bar{\mathcal{R}}_0)^+}\tilde{B}(\bar{\mathcal{R}}_0,\gamma) > \eta_{22}(\bar{\mathcal{R}}_0)$, by Theorem \ref{thm:inf:degen}(i) with $\tilde{B}_{\gamma}=\tilde{B}(\bar{\mathcal{R}}_0,\gamma)$, we obtain that $\tilde{B}(\bar{\mathcal{R}}_0,\tilde{\lambda}(\bar{\mathcal{R}}_0))=\tilde{\lambda}(\bar{\mathcal{R}}_0)$, that is, $\tilde{B}(\bar{\mathcal{R}}_0,0)=0$. It then follows that
	$
	0 =\tilde{B}(\bar{\mathcal{R}}_0,0)<\tilde{B}(\hat{\mathcal{R}}_0,0)=Q(\hat{\mathcal{R}}_0) \leq 0,
	$
	which is impossible. Therefore, thanks to \cite[Lemma 2.5]{Zhang2022SCM}, we conclude  $\lim\limits_{\theta \rightarrow 0^+} \mathcal{R}_0(\frac{1}{\theta})= \hat{\mathcal{R}}_0$.
\end{proof}

We finish this section with a brief discussion. The asymptotic behavior of the basic reproduction ratio for the general cooperative non-degenerate nonlocal dispersal system can be explored using Theorems \ref{thm:0} and \ref{thm:D:inf:non-degen}, similar to the approach in \cite{Zhang2021SIMA,Zhang2022SCM}. However, for the general partially degenerate case, further investigation is required, which we leave as a future endeavor.

\appendix
\titleformat{\section}{\large\bfseries}{\appendixname~\thesection .}{0.5em}{}
\section{Continuity of the Spectral Bound}\label{sec:conti}
The purpose of this section is to establish the continuity of the spectral bound with respect to parameters whether it degenerates or not. We use the same notations $X$, $X_+$, $A$, $K_i$ $\mathcal{A}$ and $\mathcal{K}$ as in section \ref{sec:existence}. Let $A'(x)=(a'_{ij}(x))_{l \times l}$ be a continuous matrix-valued function of $x \in \overline{\Omega}$. For each $1 \leq i \leq l$, $K_i'(x,y)$ stands for a non-negative continuous function on $\overline{\Omega} \times \overline{\Omega}$. 
Let $\mathcal{A}'$ and $\mathcal{K}'$ be two bounded linear operators on $X$ defined by
$$
[\mathcal{A}' \bm{u} ](x):=A'(x) \bm{u}(x),~ x \in \overline{\Omega}, ~ \bm{u} \in X,
$$
$$
\mathcal{K}' \bm{u}:=(\mathcal{K}_1' u_1,\cdots,\mathcal{K}_i' u_i,\cdots,\mathcal{K}_l' u_l)^T, ~ \bm{u} \in X,
$$
where
$$
[\mathcal{K}_i' v](x):
=
\int_{\Omega} K_i'(x,y) v(y) \mathrm{d} y, ~ 1\leq i \leq l,
~ x \in \overline{\Omega},~ v \in C(\overline{\Omega}).
$$
Define a bounded linear operator $\mathcal{Q}': X \rightarrow X$ by
$$
\mathcal{Q}':=\mathcal{A}'+\mathcal{K}'.
$$
Here $A'(x)$ and $K_i'(x,y)$ can be regarded as the perturbation of $A(x)$ and $K_i(x,y)$. The following lemma can be derived by standard analysis and the matrix perturbation theory (see, e.g., \cite{steward1990matrices}).

\begin{lemma}\label{lem:conti:s(A)}
	Assume that {\rm (H1)} holds. For any $\epsilon>0$, there exists $\delta>0$ such that
	\begin{itemize}
		\item[\rm (i)] If
		$
		\vert a_{ij}'(x) - a_{ij}(x) \vert \leq \delta, ~\forall x \in \overline{\Omega},~1 \leq i,j \leq l,
		$
		then
		$\Vert \mathcal{A}' - \mathcal{A} \Vert \leq \epsilon.$
		Furthermore,
		$
		\vert s(\mathcal{A}') - s(\mathcal{A}) \vert \leq \epsilon. 
		$
		\item[\rm (ii)] If $\vert K_i'(x,y) - K_i(x,y) \vert \leq \delta, ~\forall x,y \in \overline{\Omega},~1 \leq i \leq l$, then $\Vert \mathcal{K} - \mathcal{K}' \Vert \leq \epsilon$.
	\end{itemize}
\end{lemma}

We are now in a position to prove the main result of this section.
\begin{theorem}\label{thm:conti:KA}
	Assume that {\rm (H1)} holds. For any $\epsilon>0$, there exists $\delta>0$ such that 
	$\vert s(\mathcal{Q}') - s(\mathcal{Q}) \vert \leq \epsilon$, if the following statements are valid. 
	\begin{itemize}
		\item[\rm (i)] $A'(x)$ is cooperative for all $x \in \overline{\Omega}$.
		\item[\rm (ii)] $\vert a_{ij}'(x) - a_{ij}(x) \vert \leq \delta, ~\forall x \in \overline{\Omega},~1 \leq i,j \leq l$.
		\item[\rm (iii)] $\vert K_i'(x,y) - K_i(x,y) \vert \leq \delta, ~\forall x,y \in \overline{\Omega},~1 \leq i \leq l$.
	\end{itemize}
\end{theorem}

\begin{proof}	
	For any given number $\epsilon>0$, we proceed our proof into three steps.
	
	{\it Step 1:} Prove the conclusion when
	$
	a_{ij}'(x) \leq a_{ij}(x), ~\forall x \in \overline{\Omega},~1 \leq i,j \leq l
	$
	and
	$
	K_i'(x,y) \leq K_i(x,y), ~\forall x,y \in \overline{\Omega},~1 \leq i \leq l.
	$
	
	According to \cite[Theorem 1.1]{burlando1991monotonicity}, we have $s(\mathcal{Q}') \leq s(\mathcal{Q})$,
	$s(\mathcal{A}) \leq s(\mathcal{Q})$ and $s(\mathcal{A}') \leq s(\mathcal{Q}')$.
	
	By Lemma \ref{lem:conti:s(A)}, there exists $\delta_1 >0$ such that $\vert s(\mathcal{A}') - s(\mathcal{A}) \vert \leq \epsilon$ if 
	$$
	\vert a_{ij}'(x) - a_{ij}(x) \vert \leq \delta_1, ~\forall x \in \overline{\Omega},~1 \leq i,j \leq l.
	$$
	
	In the case of $s(\mathcal{Q})=s(\mathcal{A})$, $\vert s(\mathcal{Q}') -s(\mathcal{Q}) \vert \leq \epsilon$ follows from
	$$s(\mathcal{Q}') \leq s(\mathcal{Q}) = s(\mathcal{A}) \leq s(\mathcal{A}') + \epsilon \leq s(\mathcal{Q}') + \epsilon.$$
	
	In the case of $s(\mathcal{Q})>s(\mathcal{A})$, it is known that $s(\mathcal{Q})$ is an isolated eigenvalue of $\mathcal{Q}$ by $s(\mathcal{A})=s_e(\mathcal{Q})$. By the perturbation theory of isolated eigenvalue (see, e.g., \cite[Section IV.3.5]{kato1976perturbation}), there exists $\hat{\delta}_2>0$ such that if \begin{equation} \label{equ:AK_delta2}
		\Vert \mathcal{A}' - \mathcal{A} \Vert \leq \hat{\delta}_2 \text{ and } \Vert \mathcal{K}' - \mathcal{K} \Vert \leq \hat{\delta}_2,
	\end{equation} then $- \epsilon\leq s(\mathcal{Q}') - s(\mathcal{Q}) \leq 0$. In view of Lemma \ref{lem:conti:s(A)}, it is not hard to find a $\delta_2>0$ such that \eqref{equ:AK_delta2} holds if
	$$
	\vert a_{ij}'(x) - a_{ij}(x) \vert \leq \delta_2, ~\forall x \in \overline{\Omega},~1 \leq i,j \leq l,
	$$
	and
	$$
	\vert K_i'(x,y) - K_i(x,y) \vert \leq \delta_2, ~\forall x,y \in \overline{\Omega},~1 \leq i \leq l.
	$$
	The conclusion follows by choosing $\delta=\min(\delta_1,\delta_2)$.
	
	{\it Step 2:} Prove the conclusion when
	$
	a_{ij}'(x) \geq a_{ij}(x), ~\forall x \in \overline{\Omega},~1 \leq i,j \leq l
	$
	and
	$
	K_i'(x,y) \geq K_i(x,y), ~\forall x,y \in \overline{\Omega},~1 \leq i \leq l. 	
	$
	
	According to \cite[Theorem 1.1]{burlando1991monotonicity}, we have $s(\mathcal{Q}') \geq s(\mathcal{Q})$. In view of \cite[Theorem IV.3.1 and Remark IV.3.2]{kato1976perturbation}, there exists $\hat{\delta}_3>0$ such that if 
	\begin{equation}\label{equ:AK_delta3}
		\Vert \mathcal{A}' - \mathcal{A} \Vert \leq \hat{\delta}_3 \text{ and }\Vert \mathcal{K}' - \mathcal{K} \Vert \leq \hat{\delta}_3,
	\end{equation}
	then $s(\mathcal{Q}') - s(\mathcal{Q}) \leq \epsilon$, and hence, $\vert s(\mathcal{Q}') - s(\mathcal{Q}) \vert \leq \epsilon$. By Lemma \ref{lem:conti:s(A)}, we can find a $\delta_3>0$ such that \eqref{equ:AK_delta3} holds if
	$$
	\vert a_{ij}'(x) - a_{ij}(x) \vert \leq \delta_3, ~\forall x \in \overline{\Omega},~1 \leq i,j \leq l,
	$$
	and
	$$
	\vert K_i'(x,y) - K_i(x,y) \vert \leq \delta_3, ~\forall x,y \in \overline{\Omega},~1 \leq i \leq l.
	$$
	The conclusion follows by choosing $\delta=\delta_3$.
	
	{\it Step 3:} Finish the proof in general cases.
	
	Let $\overline{A}(x)=(\overline{a}_{ij}(x))_{l \times l}$ and $\underline{A}(x)=(\underline{a}_{ij}(x))_{l \times l}$ be two matrix-valued functions of $x \in \overline{\Omega}$ defined by
	$$\overline{a}_{ij}(x) = \max (a_{ij}(x),a'_{ij}(x)) \text{ and }
	\underline{a}_{ij}(x) = \min (a_{ij}(x),a'_{ij}(x)),~ \forall x \in \overline{\Omega},~1 \leq i,j \leq l. 
	$$
	Clearly, $\overline{A}(x)$ and $\underline{A}(x)$ are still cooperative for all $x \in \overline{\Omega}$.
	Define
	$$
	\overline{K}_i(x,y):=\max(K_i(x,y),K_i'(x,y)), ~\forall x,y \in \overline{\Omega},~1 \leq i \leq l,
	$$
	and
	$$
	\underline{K}_i(x,y):=\min(K_i(x,y),K_i'(x,y)), ~\forall x,y \in \overline{\Omega},~1 \leq i \leq l.
	$$
	
	Let $\overline{\mathcal{A}}$, $\underline{\mathcal{A}}$, $\overline{\mathcal{K}}$ and $\underline{\mathcal{K}}$ be four bounded linear operators on $X$ defined by:
	$$
	[\overline{\mathcal{A}} \bm{u} ](x):=\overline{A}(x) \bm{u}(x),~
	[\underline{\mathcal{A}} \bm{u} ](x):=\underline{A}(x) \bm{u}(x),~ \forall x \in \overline{\Omega}, ~ \bm{u} \in X,
	$$
	$$
	\overline{\mathcal{K}} \bm{u}:=(\overline{\mathcal{K}}_1 u_1, \overline{\mathcal{K}}_2 u_2,\cdots,\overline{\mathcal{K}}_i u_i,\cdots,\overline{\mathcal{K}}_l u_l)^T, ~\forall \bm{u} \in X,
	$$
	$$
	\underline{\mathcal{K}} \bm{u}:=(\underline{\mathcal{K}}_1 u_1, \underline{\mathcal{K}}_2 u_2,\cdots,\underline{\mathcal{K}}_i u_i,\cdots,\underline{\mathcal{K}}_l u_l)^T, ~\forall \bm{u} \in X,
	$$
	where
	$$
	[\overline{\mathcal{K}}_i v](x):
	=
	\int_{\Omega} \overline{K}_i(x,y) v(y) \mathrm{d} y, ~\forall 1 \leq i \leq l,
	~ x \in \overline{\Omega},~ v \in C(\overline{\Omega}),
	$$
	$$
	[\underline{\mathcal{K}}_i v](x):
	=
	\int_{\Omega} \underline{K}_i(x,y) v(y) \mathrm{d} y, ~\forall 1 \leq i \leq l,	~ x \in \overline{\Omega},~ v \in C(\overline{\Omega}).
	$$
	Define $\overline{\mathcal{Q}}$ and $\underline{\mathcal{Q}}$ on $X$ by
	$$
	\overline{\mathcal{Q}}:= \overline{\mathcal{A}}+\overline{\mathcal{K}}, \qquad
	\underline{\mathcal{Q}}:=\underline{\mathcal{A}}+\underline{\mathcal{K}}.
	$$
	
	According to \cite[Theorem 1.1]{burlando1991monotonicity}, it is easy to see 
	$$s(\underline{\mathcal{Q}})\leq s(\mathcal{Q}) \leq s(\overline{\mathcal{Q}}) \text{ and }
	s(\underline{\mathcal{Q}})\leq s(\mathcal{Q}') \leq s(\overline{\mathcal{Q}}).$$
	
	Choosing $\delta = \min(\delta_1,\delta_2,\delta_3)$, we have 
	$$
	\vert a_{ij}'(x) - a_{ij}(x) \vert \leq \delta , ~\forall x \in \overline{\Omega},~1 \leq i,j \leq l,
	$$
	and
	$$
	\vert K_i'(x,y) - K_i(x,y) \vert \leq \delta \leq \min (\delta_2,\delta_3), ~\forall x,y \in \overline{\Omega},~1 \leq i \leq l.
	$$
	This implies that
	$$
	\max(\vert \overline{a}_{ij}(x) - a_{ij}(x) \vert,\vert \underline{a}_{ij}(x) - a_{ij}(x) \vert) \leq \delta, ~\forall x \in \overline{\Omega},~1 \leq i,j \leq l,
	$$
	and
	$$
	\max(\vert \overline{K}_i(x,y) - K_i(x,y) \vert,\vert \underline{K}_i(x,y) - K_i(x,y) \vert)
	\leq \min (\delta_2,\delta_3), ~\forall x,y \in \overline{\Omega},~1 \leq i \leq l.
	$$
	Consequently,
	$
	\Vert \underline{\mathcal{K}} - \mathcal{K} \Vert \leq \hat{\delta}_2
	$,
	$
	\Vert \underline{\mathcal{A}} - \mathcal{A} \Vert \leq \hat{\delta}_2
	$,
	$
	\Vert \overline{\mathcal{K}} - \mathcal{K} \Vert \leq \hat{\delta}_3
	$,
	and
	$
	\Vert \overline{\mathcal{A}} - \mathcal{A} \Vert \leq \hat{\delta}_3.
	$
	Therefore, 
	$s(\underline{\mathcal{Q}}) - s(\mathcal{Q}) \geq -\epsilon\text{ and }
	s(\overline{\mathcal{Q}}) - s(\mathcal{Q}) \leq \epsilon$
	follow from the previous two steps. Finally, we conclude that
	$$
	-\epsilon \leq
	s(\underline{\mathcal{Q}}) - s(\mathcal{Q})\leq
	s(\mathcal{Q}') - s(\mathcal{Q})\leq
	s(\overline{\mathcal{Q}}) - s(\mathcal{Q})\leq
	\epsilon,
	$$
	which completes the proof. 
\end{proof}

{\bf \noindent Acknowledgments.} I would like to thank Profs. Xiao-Qiang Zhao and Xing Liang for helpful discussions during the preparation of this work. I am also grateful to the anonymous referees for their careful reading and helpful comments, which led to an improvement of my original manuscript. This research is supported by the National Natural Science Foundation of China (12171119) and the Fundamental Research Funds for the Central Universities (GK202304029, GK202306003, GK202402004).




\begin{thebibliography}{10}
	
	\bibitem{allen2008asymptotic}
	{\sc L.~J. Allen, B.~M. Bolker, Y.~Lou, and A.~L. Nevai}, {\em Asymptotic
		profiles of the steady states for an {SIS} epidemic reaction-diffusion
		model}, Discrete Contin. Dynam. Systems, 21 (2008), pp.~1--20.
	
	\bibitem{allen2007Asymptotic}
	{\sc L.~J.~S. Allen, B.~M. Bolker, Y.~Lou, and A.~L. Nevai}, {\em Asymptotic
		profiles of the steady states for an {SIS} epidemic patch model}, SIAM J.
	Appl. Math., 67 (2007), pp.~1283--1309.
	
	\bibitem{andreu2010nonlocal}
	{\sc F.~Andreu-Vaillo, J.~M. Maz{\'o}n, J.~D. Rossi, and J.~J. Toledo-Melero},
	{\em Nonlocal Diffusion Problems}, Mathematical Surveys and Monographs,
	American Mathematical Society, Providence, Rhode Island, 2010.
	
	\bibitem{bao2020propagation}
	{\sc X.~Bao and W.-T. Li}, {\em Propagation phenomena for partially degenerate
		nonlocal dispersal models in time and space periodic habitats}, Nonlinear
	Anal. Real World Appl., 51 (2020), 102975.
	
	\bibitem{bao2017criteria}
	{\sc X.~Bao and W.~Shen}, {\em Criteria for the existence of principal
		eigenvalues of time periodic cooperative linear systems with nonlocal
		dispersal}, Proc. Amer. Math. Soc., 145 (2017), pp.~2881--2894.
	
	\bibitem{bates2007existence}
	{\sc P.~W. Bates and G.~Zhao}, {\em Existence, uniqueness and stability of the
		stationary solution to a nonlocal evolution equation arising in population
		dispersal}, J. Math. Anal. Appl., 332 (2007), pp.~428--440.
	
	\bibitem{berestycki2016definition}
	{\sc H.~Berestycki, J.~Coville, and H.-H. Vo}, {\em On the definition and the
		properties of the principal eigenvalue of some nonlocal operators}, J. Funct.
	Anal., 271 (2016), pp.~2701--2751.
	
	\bibitem{berestycki2002front}
	{\sc H.~Berestycki and F.~Hamel}, {\em Front propagation in periodic excitable
		media}, Comm. Pure Appl. Math., 55 (2002), pp.~949--1032.
	
	\bibitem{berman1994nonnegative}
	{\sc A.~Berman and R.~J. Plemmons}, {\em Nonnegative Matrices in the
		Mathematical Sciences}, SIAM, Philadelphia, 1994.
	
	\bibitem{burger1988perturbations}
	{\sc R.~B{\"u}rger}, {\em Perturbations of positive semigroups and applications
		to population genetics}, Math. Z., 197 (1988), pp.~259--272.
	
	\bibitem{burlando1991monotonicity}
	{\sc L.~Burlando}, {\em Monotonicity of spectral radius for positive operators
		on ordered {Banach} spaces}, Arch. Math. (Basel), 56 (1991), pp.~49--57.
	
	\bibitem{cantrell2004spatial}
	{\sc R.~S. Cantrell and C.~Cosner}, {\em Spatial Ecology via Reaction-Diffusion
		Equations}, Wiley, Chichester, 2004.
	
	\bibitem{coville2010simple}
	{\sc J.~Coville}, {\em On a simple criterion for the existence of a principal
		eigenfunction of some nonlocal operators}, J. Differential Equations, 249
	(2010), pp.~2921--2953.
	
	\bibitem{coville2008existence}
	{\sc J.~Coville, J.~D{\'a}vila, and S.~Mart{\'\i}nez}, {\em Existence and
		uniqueness of solutions to a nonlocal equation with monostable nonlinearity},
	SIAM J. Math. Anal., 39 (2008), pp.~1693--1709.
	
	\bibitem{dancer2009principal}
	{\sc E.~N. Dancer}, {\em On the principal eigenvalue of linear cooperating
		elliptic systems with small diffusion}, J. Evolut. Eqns., 9 (2009),
	pp.~419--428.
	
	\bibitem{deimling1985nonlinear}
	{\sc K.~Deimling}, {\em Nonlinear Functional Analysis}, Springer-Verlag,
	Berlin, Heidelberg, 1985.
	
	\bibitem{du2012analysis}
	{\sc Q.~Du, M.~Gunzburger, R.~B. Lehoucq, and K.~Zhou}, {\em Analysis and
		approximation of nonlocal diffusion problems with volume constraints}, SIAM
	Rev., 54 (2012), pp.~667--696.
	
	\bibitem{edmunds1972non}
	{\sc D.~Edmunds, A.~Potter, and C.~Stuart}, {\em Non-compact positive
		operators}, Proc. R. Soc. Lond. Ser. A Math. Phys. Eng. Sci., 328 (1972),
	pp.~67--81.
	
	\bibitem{fife2003some}
	{\sc P.~Fife}, {\em Some nonclassical trends in parabolic and parabolic-like
		evolutions}, in M. Kirkilionis, S. Kr{\"o}mker, R. Ran- nacher, F. Tomi
	(Eds.), Trends in nonlinear analysis, Springer, Berlin, Heidelberg, 2003,
	pp.~153--191.
	
	\bibitem{gao2019travel}
	{\sc D.~Gao}, {\em Travel frequency and infectious diseases}, SIAM J. Appl.
	Math, 79 (2019), pp.~1581--1606.
	
	\bibitem{gao2020fast}
	{\sc D.~Gao and C.-P. Dong}, {\em Fast diffusion inhibits disease outbreaks},
	Proc. Amer. Math. Soc., 148 (2020), pp.~1709--1722.
	
	\bibitem{hale1987varying}
	{\sc J.~K. Hale and C.~Rocha}, {\em Varying boundary conditions with large
		diffusivity}, J. Math. Pures Appl., 66 (1987), pp.~139--158.
	
	\bibitem{hsu2015pivotal}
	{\sc S.-B. Hsu, J.~L{\'o}pez-G{\'o}mez, L.~Mei, and F.-B. Wang}, {\em A pivotal
		eigenvalue problem in river ecology}, J. Differential Equations, 259 (2015),
	pp.~2280--2316.
	
	\bibitem{hsu2011dynamics}
	{\sc S.-B. Hsu, F.-B. Wang, and X.-Q. Zhao}, {\em Dynamics of a periodically
		pulsed bio-reactor model with a hydraulic storage zone}, J. Dynam.
	Differential Equations, 23 (2011), pp.~817--842.
	
	\bibitem{huang2016r0}
	{\sc Q.~Huang, Y.~Jin, and M.~A. Lewis}, {\em {$R_0$} analysis of a
		benthic-drift model for a stream population}, SIAM J. Appl. Dyn. Syst., 15
	(2016), pp.~278--321.
	
	\bibitem{hutson2003evolution}
	{\sc V.~Hutson, S.~Martinez, K.~Mischaikow, and G.~T. Vickers}, {\em The
		evolution of dispersal}, J. Math. Biol., 47 (2003), pp.~483--517.
	
	\bibitem{hutson2008spectral}
	{\sc V.~Hutson, W.~Shen, and G.~Vickers}, {\em Spectral theory for nonlocal
		dispersal with periodic or almost-periodic time dependence}, R. Mount. J.
	Math., (2008), pp.~1147--1175.
	
	\bibitem{jiang2004saddle}
	{\sc J.~Jiang, X.~Liang, and X.-Q. Zhao}, {\em Saddle-point behavior for
		monotone semiflows and reaction--diffusion models}, J. Differential
	Equations, 203 (2004), pp.~313--330.
	
	\bibitem{jin2011seasonal}
	{\sc Y.~Jin and M.~A. Lewis}, {\em Seasonal influences on population spread and
		persistence in streams: critical domain size}, SIAM J. Appl. Math., 71
	(2011), pp.~1241--1262.
	
	\bibitem{kao2010random}
	{\sc C.-Y. Kao, Y.~Lou, and W.~Shen}, {\em Random dispersal vs. nonlocal
		dispersal}, Discrete Contin. Dynam. Systems, 26 (2010), pp.~551--596.
	
	\bibitem{kato1976perturbation}
	{\sc T.~Kato}, {\em Perturbation Theory for Linear Operators, Classics in
		Mathematics}, Reprint of the 1980 edition, Springer-Verlag, Berlin,
	Heidelberg, 1995.
	
	\bibitem{lam2016asymptotic}
	{\sc K.~Y. Lam and Y.~Lou}, {\em Asymptotic behavior of the principal
		eigenvalue for cooperative elliptic systems and applications}, J. Dynam.
	Differential Equations, 28 (2016), pp.~29--48.
	
	\bibitem{lancaster1985matrices}
	{\sc P.~Lancaster and M.~Tismenetsky}, {\em The Theory of Matrices: with
		Applications}, Academic Press, San Diego, California, second~ed., 1985.
	
	\bibitem{li2017eigenvalue}
	{\sc F.~Li, J.~Coville, and X.~Wang}, {\em On eigenvalue problems arising from
		nonlocal diffusion models}, Discrete Contin. Dynam. Systems, 37 (2017),
	pp.~879--903.
	
	\bibitem{liang2017principal}
	{\sc X.~Liang, L.~Zhang, and X.-Q. Zhao}, {\em The principal eigenvalue for
		degenerate periodic reaction-diffusion systems}, SIAM J. Math. Anal., 49
	(2017), pp.~3603--3636.
	
	\bibitem{liang2019principal}
	{\sc X.~Liang, L.~Zhang, and X.-Q. Zhao}, {\em The principal eigenvalue for
		periodic nonlocal dispersal systems with time delay}, J. Differential
	Equations, 266 (2019), pp.~2100--2124.
	
	\bibitem{liang2007asymptotic}
	{\sc X.~Liang and X.-Q. Zhao}, {\em Asymptotic speeds of spread and traveling
		waves for monotone semiflows with applications}, Comm. Pure Appl. Math., 60
	(2007), pp.~1--40.
	
	\bibitem{lutscher2006effects}
	{\sc F.~Lutscher, M.~A. Lewis, and E.~McCauley}, {\em Effects of heterogeneity
		on spread and persistence in rivers}, Bull. Math. Biol., 68 (2006),
	pp.~2129--2160.
	
	\bibitem{mckenzie2012r_0}
	{\sc H.~Mckenzie, Y.~Jin, J.~Jacobsen, and M.~Lewis}, {\em {$R_0$} analysis of
		a spatiotemporal model for a stream population}, SIAM J. Appl. Dyn. Syst., 11
	(2012), pp.~567--596.
	
	\bibitem{meyer2000matrix}
	{\sc C.~D. Meyer}, {\em Matrix analysis and applied linear algebra}, vol.~71,
	SIAM, Philadelphia, 2000.
	
	\bibitem{ni2011mathematics}
	{\sc W.-M. Ni}, {\em The Mathematics of Diffusion}, vol.~82, SIAM,
	Philadelphia, 2011.
	
	\bibitem{nussbaum1981eigenvectors}
	{\sc R.~Nussbaum}, {\em Eigenvectors of nonlinear positive operators and the
		linear {Krein-Rutman} theorem}, Fixed Point Theory, 886 (1981), pp.~309--330.
	
	\bibitem{pachepsky2005persistence}
	{\sc E.~Pachepsky, F.~Lutscher, R.~Nisbet, and M.~Lewis}, {\em Persistence,
		spread and the drift paradox}, Theor. Popul. Biol., 67 (2005), pp.~61--73.
	
	\bibitem{reed1980methods}
	{\sc M.~Reed and B.~Simon}, {\em Methods of Modern Mathematical Physics. vol.
		1. Functional Analysis}, Academic, New York, 1980.
		
	\bibitem{schaefer1960some}
	{\sc H.~Schaefer}, {\em Some spectral properties of positive linear operators},
	Pacific J. Math., 10 (1960), pp.~1009--1019.
	
	\bibitem{schechter1971principles}
	{\sc M.~Schechter}, {\em Principles of Functional Analysis}, vol.~2, Academic,
	New York, 1971.
	
	\bibitem{schneider1984theorems}
	{\sc H.~Schneider}, {\em Theorems on m-splittings of a singular m-matrix which
		depend on graph structure}, Linear Algebra Appl., 58 (1984), pp.~407--424.
	
	\bibitem{shenxie2016spectral}
	{\sc W.~Shen and X.~Xie}, {\em Spectral theory for nonlocal dispersal operators
		with time periodic indefinite weight functions and applications}, Discrete
	Contin. Dyn. Syst. Ser. B, 22 (2017), pp.~1023--1047.
	
	\bibitem{shen2010spreading}
	{\sc W.~Shen and A.~Zhang}, {\em Spreading speeds for monostable equations with
		nonlocal dispersal in space periodic habitats}, J. Differential Equations,
	249 (2010), pp.~747--795.
	
	\bibitem{shen2012Stationary}
	{\sc W.~Shen and A.~Zhang}, {\em Stationary solutions and spreading speeds of
		nonlocal monostable equations in space periodic habitats}, Proc. Amer. Math.
	Soc., 140 (2012), pp.~1681--1696.
	
	\bibitem{smith2008monotone}
	{\sc H.~L. Smith}, {\em Monotone Dynamical Systems: an Introduction to the
		Theory of Competitive and Cooperative Systems}, no.~41 in Mathematical
	Surveys and Monographs, American Mathematical Society, Providence, 2008.
	
	\bibitem{steward1990matrices}
	{\sc G.~Steward and J.~Sun}, {\em Matrix perturbation theory}, Academic Press,
	Boston, 1990.
	
	\bibitem{Shu2020JMPA}
	{\sc H.~Shu, Z.~Ma, X.-S. Wang, and L.~Wang}, {\em Viral diffusion and
		cell-to-cell transmission: mathematical analysis and simulation study}, J.
	Math. Pures Appl. (9), 137 (2020), pp.~290--313.
	
	\bibitem{Su2023JDE}
	{\sc Y.-H. Su, X.~Wang, and T.~Zhang}, {\em Principal spectral theory and
		variational characterizations for cooperative systems with nonlocal and
		coupled diffusion}, J. Differential Equations, 369 (2023), pp.~94--114.
		
	\bibitem{thieme2009spectral}
	{\sc H.~R. Thieme}, {\em Spectral bound and reproduction number for
		infinite-dimensional population structure and time heterogeneity}, SIAM J.
	Appl. Math., 70 (2009), pp.~188--211.
	
	\bibitem{Wang2014AA}
	{\sc F.-B. Wang, Y.~Huang, and X.~Zou}, {\em Global dynamics of a {PDE} in-host
		viral model}, Appl. Anal., 93 (2014), pp.~2312--2329.
		
	
	\bibitem{wang2018global}
	{\sc J.-B. Wang, W.-T. Li, and J.-W. Sun}, {\em Global dynamics and spreading
		speeds for a partially degenerate system with non-local dispersal in periodic
		habitats}, Proc. Edinb. Math. Soc., 148 (2018), pp.~849--880.
	
	\bibitem{Wang2016JMAA}
	{\sc J.~Wang, J.~Yang, and T.~Kuniya}, {\em Dynamics of a {PDE} viral infection
		model incorporating cell-to-cell transmission}, J. Math. Anal. Appl., 444
	(2016), pp.~1542--1564.
	
	\bibitem{wang2012basic}
	{\sc W.~Wang and X.-Q. Zhao}, {\em Basic reproduction numbers for
		reaction-diffusion epidemic models}, SIAM J. Appl. Dyn. Syst., 11 (2012),
	pp.~1652--1673.
	
	\bibitem{yang2019dynamics}
	{\sc F.-Y. Yang, W.-T. Li, and S.~Ruan}, {\em Dynamics of a nonlocal dispersal
		sis epidemic model with neumann boundary conditions}, J. Differential
	Equations, 267 (2019), pp.~2011--2051.
	
	\bibitem{Zhang2021SIMA}
	{\sc L.~Zhang and X.-Q. Zhao}, {\em Asymptotic behavior of the basic
		reproduction ratio for periodic reaction-diffusion systems}, SIAM J. Math.
	Anal., 53 (2021), pp.~6873--6909.
	
	\bibitem{Zhang2022SCM}
	{\sc L.~Zhang and X.-Q. Zhao}, {\em Asymptotic behavior of the principal
		eigenvalue and the basic reproduction ratio for periodic patch models}, Sci.
	China Math., 65 (2022), pp.~1363--1382.
	
	\bibitem{zhao2017dynamical}
	{\sc X.-Q. Zhao}, {\em Dynamical Systems in Population Biology}, Springer, New
	York, second~ed., 2017.
	
\end{thebibliography}

\end{document}